\documentclass[a4paper]{article}
\usepackage{a4wide}
\usepackage{amsmath}
\usepackage{amssymb,xfrac}
\usepackage{authblk}

\DeclareSymbolFont{forpolishl}{T1}{cmr}{m}{n}
\DeclareMathSymbol{\mathrmL}{0}{forpolishl}{'212}
\DeclareSymbolFont{forpolishbl}{T1}{cmr}{bx}{n}
\DeclareMathSymbol{\mathbfL}{0}{forpolishbl}{'212}
\DeclareMathSymbol{\Phalf}{0}{forpolishl}{'275}
\DeclareMathSymbol{\Nabla}{2}{symbols}{"72}
\DeclareMathSymbol{\minetriangle}{1}{symbols}{"35}
\DeclareSymbolFont{forpolishblit}{T1}{cmr}{bx}{it}
\DeclareMathSymbol{\mathbfitL}{0}{forpolishblit}{'212}

\newcommand{\f}{\ensuremath{\varphi}}
\newcommand{\p}{\ensuremath{\psi}}
\newcommand{\x}{\ensuremath{\chi}}

\newcommand{\1}{\ensuremath{\overline{1}}}
\newcommand{\0}{\ensuremath{\overline{0}}}

\newcommand{\tuple}[1]{\ensuremath{\langle{#1}\rangle}}  % fixed size of angle brackets
\newcommand{\vectn}[1]{\ensuremath{#1_1,\ldots,#1_n}}    % list of 1..n variables
\newcommand{\tri}{\triangle}

\newcommand{\thalf}{\ensuremath{{\mathchoice
			{\mbox{\hspace{-0.3pt}}\raisebox{.37ex}{$\scriptstyle\frac12$}}
			{\mbox{\hspace{-0.3pt}}\raisebox{.37ex}{$\scriptstyle\frac12$}}
			{\mbox{\hspace{-0.5pt}}\raisebox{.25ex}{$\scriptscriptstyle\frac12$}}
			{\mbox{\hspace{-0.5pt}}\raisebox{.1ex}{$\scriptscriptstyle\frac12$}}
		}}}
		
		\newcommand{\lang}[1]{\ensuremath{\mathcal #1}}

		\newcommand{\fmVL}[2]{\ensuremath{\mathit{Fm}_{{\lang{#2}}}^{{#1}}}}
		\newcommand{\fmL}[1]{\fmVL{}{#1}}
		\newcommand{\fm}{\fmL{L}}

		\newcommand{\logic}[1]{\ensuremath{\mathrm #1}}

		\newcommand{\alg}[1]{{\ensuremath{\boldsymbol {\mathit{#1}}}}}

		\newcommand{\PL}{{\ensuremath{\mathrm{PL}}}}

		\newcommand{\G}{{\ensuremath{\mathrm{G}}}}

		\newcommand{\MV}{{\ensuremath{\mathrm{MV}}}}
		
		\newcommand{\Luk}{\ensuremath{{\mathrmL}}}
		\newcommand{\LPi}{\ensuremath{{\Luk\mathrm{\Pi}}}}
		\newcommand{\LPih}{\ensuremath{{\Luk\mathrm{\Pi}\thalf}}}
		\newcommand{\prl}{{\logic{P}\mathrmL}}

		\def\plabel#1{\@bsphack
			\protected@write\@auxout{}%
			{\string\newlabel{#1}{{$($\@propprefix\@currentlabel$)$}{\thepage}}}%
			\@esphack}

		\newcounter{qcounter}                                           %  new list environments

\usepackage{amsthm}
\theoremstyle{plain}
\newtheorem{theorem}{Theorem}[section]
\newtheorem{lemma}[theorem]{Lemma}
\newtheorem{proposition}[theorem]{Proposition}

\theoremstyle{definition}

\newtheorem{definition}[theorem]{Definition}
\newtheorem{convention}[theorem]{Convention}
\newtheorem{example}[theorem]{Example}

\newtheorem*{Remark*}{Remark}
\theoremstyle{remark}
\newtheorem{remark}[theorem]{Remark}
\newtheorem*{remark*}{Remark}

\usepackage{enumitem}

\usepackage{color}

\newcommand{\Rc}{\mathbb{R}}

\newcommand{\Gcal}{\ensuremath{\mathcal{G}}}

\newcommand{\Q}{\mathbb{Q}}

\renewcommand{\emptyset}{\varnothing}
\renewcommand{\phi}{\varphi}

\newcommand{\Ex}{\mathsf{E}}
\newcommand{\pr}{\mathbf{pr}}

\newcommand{\pd}{\ensuremath{{\mathtt{ProbDistr}}}}

\providecommand{\abs}[1]{\lvert#1\rvert}

\newcommand{\ES}[2]{#1{\downharpoonright}{#2}}

\renewcommand{\prl}{{\logic{P\mathrmL}}} 
\renewcommand{\vec}[1]{\ensuremath{\boldsymbol{#1}}}

\newcommand{\bigsymbolsizechoice}[1]{\mathchoice
    {\mbox{\raisebox{-2pt}{\LARGE{$#1$}}}}
    {\mbox{\raisebox{-1pt}{\large{$#1$}}}}
    {\mbox{\raisebox{-1pt}{\normalsize{$#1$}}}}
    {\mbox{\raisebox{-1pt}{\scriptsize{$#1$}}}}}
\newcommand{\bigop}[1]{\ensuremath{\mathop{\bigsymbolsizechoice{#1}}\limits}}

\begin{document}
%
%\frontmatter          % for the preliminaries
%
\pagestyle{headings}  % switches on printing of running heads

\title{Representing Strategic Games and Their Equilibria\\ in Many-Valued Logics\thanks{%
Supported by projects P402/12/1309 of the Czech Science Foundation; 7AMB13AT014 of the Ministry of Education, Youth, and Sports of the Czech Republic; RVO 67985807;
and Austrian Science Fund (FWF) project P25417--G15.
The work on a major revision of the paper was supported by the joint
project of the Austrian Science Fund (No.~I1897--N25) and the Czech Science
Foundation (No.~15--34650L). The authors are indebted to the anonymous referees for their invaluable comments which led to a major improvement of the paper.}
}

%\titlerunning{Representing Strategic Games}

\author[1,4]{Libor B\v{e}hounek}
\author[1]{Petr Cintula}
\author[2]{Chris Ferm\"uller}
\author[3]{Tom\'a\v{s} Kroupa}
\affil[1]{\small Institute of Computer Science, Czech Academy of Sciences,
    Pod Vod\'arenskou v\v{e}\v{z}\'i~2, 182~07~Prague, Czech Republic}
\affil[2]{Institute of Computer Languages 185.2, Vienna University of Technology,
    Favoritenstra\ss e~9--11, 1040~Vienna, Austria}
\affil[3]{Institute of Information Theory and Automation, Czech Academy of Sciences,
    Pod Vod\'arenskou v\v{e}\v{z}\'i~4, 182~08~Prague, Czech Republic}
\affil[4]{Institute for Research and Applications of Fuzzy Modeling, University of Ostrava, NSC IT4Innovations,
    30.~dubna~22, 701~03~Ostrava~1, Czech Republic}

\maketitle              % typeset the title of the contribution

\vspace{-2ex}

\begin{abstract}
We introduce the notion of logical $\alg{A}$-games for a fairly general class of algebras $\alg{A}$ of real truth-values.
This concept generalizes the Boolean games
%FFV a more careful formulation taking into account different strands of logical games?
{of Harrenstein et al.}\
as well  as the recently defined {\L}ukasiewicz games of Marchioni and Wooldridge.
We demonstrate that a wide range of strategic $n$-player games can be represented
as logical $\alg{A}$-games. Moreover we show how to construct, under rather general conditions,
propositional formulas in the language of  $\alg{A}$ that correspond to pure
and mixed Nash equilibria of logical $\alg{A}$-games.

\end{abstract}

\section{Introduction}

Various types of connections between logic and game theory
increasingly receive attention in the literature. (We refer to~\cite{Benthem:LogicInGames} for a  recent
monograph devoted to several aspects of this topic.)
This paper is a contribution to a special line of research that has been initiated by the
 introduction of the concept of a Boolean game in~\cite{Harrenstein-HMW:BooleanGames}.
Originally, Boolean games have been introduced as two-person zero-sum extensive-form
games. However, here we follow the bulk of literature that views
Boolean games as special strategic $n$-player games, where each player's
payoff function is expressed by a~classical  propositional  (i.e., Boolean) formula and
her strategies consist in truth-value assignments to a subset
of the propositional variables occurring in the payoff functions.
The~focus on classical formulas severely limits the scope of
strategic games that can be represented in this format. In particular, one is
often interested in finite games with more than just two possible payoff values,
which entails that the payoff functions cannot be identified with Boolean formulas.
For example, nearly all of the well-known strategic games that are usually represented
by a $2\times 2$ payoff matrix,
such as the Prisoner's Dilemma, Chicken, the Coordination Game, etc.\
(see, e.g., \cite{Dresher:GamesStrategy,Maschler-Solan-Zamir:GameTheory,Osborne:IntroductionGameTheory,Owen:GameTheory}),
fall into this category.
This fact has motivated
Marchioni and Wooldridge~\cite{Marchioni-Wooldridge:LukasiewiczGames,Marchioni-Wooldridge:LukasiewiczGamesTOCL}
to generalize Boolean games to so-called {\L}ukasiewicz games, where
the payoff functions are represented by
formulas of {an appropriate (finite or infinite)}
{\L}ukasiewicz logic and the strategies are assignments of
the logic's truth values to relevant variables.
A number of well-known strategic games, or at least modified variants of them,
are representable as {\L}ukasiewicz games.
Moreover, it is shown in~\cite{Marchioni-Wooldridge:LukasiewiczGames}
that for any given finite {\L}ukasiewicz game~$\Gcal$ there is a~propositional formula $\f_{\Gcal}$
which is satisfiable in the corresponding {\L}ukasiewicz logic iff $\Gcal$ has a pure Nash equilibrium.
In~\cite{Marchioni-Wooldridge:LukasiewiczGamesTOCL} this result is generalized to
infinite{-}valued {\L}ukasiewicz logic.
However, instead of directly expressing Nash equilibria by propositional formulas,
a detour via classical first-order theories of corresponding algebras is employed.

The main aim of this paper is to generalize and expand the approach of
Marchioni and Wooldridge in at least three different aspects:
\begin{itemize}
\item[(1)] We show that the restriction to {\L}ukasiewicz logics as the underlying formalism for the
     representation of games is neither necessary nor convenient. In fact, rather than just proposing
     additional many-valued logic as possible representation formalisms, we aim at
     identifying \emph{general conditions} that are sufficient for representing wide classes of
     games as well as expressing their Nash equilibria.
\item[(2)] We remove yet another, quite different limitation of {\L}ukasiewicz games:
  Marchioni and Wooldridge~\cite{Marchioni-Wooldridge:LukasiewiczGames,Marchioni-Wooldridge:LukasiewiczGamesTOCL} identify the set
  of strategies of a given player with the set of \emph{all} assignments of truth values
  to the variables controlled by that player. While this makes sense for Boolean games, it amounts to an unnecessary,
  and in fact rather obstructive, restriction in the many-valued setting.
  As we will demonstrate, by simply using \emph{subsets} of all possible assignments to
  represent a player's strategies, not only a wider class of games, but in particular
  \emph{all finite} strategic games become representable as logical games.
\item[(3)] So far, only the characterization of the existence of \emph{pure} Nash equilibria
   by logical formulas has been considered in the literature. We will show that, for sufficiently
   expressive logics, also \emph{mixed} Nash equilibria can be characterized by propositional
   formulas. The emphasis
   here is not on their existence, which in all cases relevant here is guaranteed by
   Nash's Theorem, but on the fact that we may use propositional variables to represent
   probability distributions and thus obtain a one-to-one correspondence between the assignments
   satisfying a particular formula and the mixed equilibria of the games in question.
\end{itemize}
Overall, we demonstrate that many-valued logics provide an adequate setting for
the formal representation of large classes of strategic games, in particular of
all finite strategic games.
This includes the reduction of questions about pure as well as mixed Nash equilibria
into questions about the satisfiability of appropriate propositional formulas. While
our results immediately suggest straightforward algorithms for checking (the existence of)
equilibria, it remains to be seen which further benefits can be reaped from
our logical approach to the representation of strategic $n$-person games.

Let us mention some further features that distinguish our approach.
While we talk about ``logical games'', the central reference is actually to a wide class of
so-called standard algebras, i.e., algebras over (subsets of) the real unit interval $[0,1]$.
For our purpose, the distinction between formulas and terms over an algebra is in fact immaterial.
As already mentioned, in~\cite{Marchioni-Wooldridge:LukasiewiczGamesTOCL}
the existence of pure Nash equilibria is not expressed directly by propositional formulas,
but only indirectly via classical first-order theories of particular algebras. We will show
that this detour is unnecessary.
A further item that deserves to be emphasized right away concerns the very concept
of representing a given strategic game as a logical game (with respect
to a given algebra). Note that the notion of representability is only implicit
in~\cite{Marchioni-Wooldridge:LukasiewiczGames,Marchioni-Wooldridge:LukasiewiczGamesTOCL} as well as in the literature
on Boolean games. (\cite{Dunne-Hoek:BooleanGames} discusses succinct representability of
Boolean games, but does not refer to general strategic games.)
By making representability explicit, we
disambiguate this somewhat vague notion and are able to formally characterize
the scope of representable games. Moreover, this move supports the identification
of different conditions on expressibility that are sufficient to express Nash equilibria
at various levels of succinctness.

\pagebreak

We emphasize that the aim of this paper is to demonstrate that many-valued logics provide a~versatile and very general
tool for the formalization of strategic games. Once appropriate notions and conditions
are identified, checking that the corresponding logical representations are indeed adequate
is routine and consequently left to the reader in most cases.

The paper is organized as follows. In Section~\ref{sec:preliminaries} we present the basic concepts and terminology used in later sections: Subsection~\ref{ssec:logics} fixes some notions regarding algebras and logics  over (subsets of) the real unit interval $[0,1]$. Subsection~\ref{ssec:strategic_games} reviews basic game-theoretic notions and illustrates these by
presenting a number of concrete examples that are taken up in later sections.
Section~\ref{sec:representing} introduces the concept of a logical game with respect to an arbitrary standard algebra. We demonstrate that a wide class of ordinary strategic games
can be represented as logical games and provide corresponding examples.
In Section~\ref{sec:pureNE} we show how (the existence of) pure Nash equilibria in logical games can be expressed by propositional formulas under rather weak conditions.  Section~\ref{sec:mixedNE} is devoted to the construction of propositional formulas that correspond to mixed Nash equilibria for suitable classes of logical games. We conclude with Section~\ref{sec:conclusion} containing a short summary and some remarks on possible directions of future research, with emphasis on dealing with infinite games.

\section{Preliminary notions}\label{sec:preliminaries}

\subsection{Logics and algebras} \label{ssec:logics}

We will work with logics expressed in various propositional languages.
A (propositional) language $\lang{L}$  is understood as a collection of connectives equipped with arities.
The corresponding set $\fm$ of propositional formulas is defined over
a countably infinite set of propositional variables as usual:
\begin{itemize}
\item $\fm$ contains all propositional variables.
\item $\circ(\vectn{\f}) \in \fm$ if  $\vectn{\f} \in \fm$ and $\circ$ is an $n$-ary connective in~$\lang{L}$.
    (We will use infix notation for familiar binary connectives.
    Nullary connectives are also called \emph{truth constants.})
\item Nothing else is in $\fm$.
\end{itemize}

In this paper we consider special many-valued logics, each of which is determined by a single particular algebra of truth degrees;
proof systems will play no role here. (The interested reader can find information about
deductive aspects of the kinds of many-valued logics treated here in the handbook chapter~\cite{Metcalfe:Handbook}.)
We will always assume that each language contains at least three binary
connectives $\wedge$, $\vee$, and~$\rightarrow$.
We will identify propositional languages with algebraic types, connectives with operation symbols,
and formulas with terms of the corresponding algebra.

\begin{definition} \label{def:A}
A \emph{standard algebra} $\alg{A}$ (of truth degrees) in a language $\lang{L}_\alg{A}$
is a tuple $\tuple{A,\tuple{\circ^{\alg{A}}}_{{\circ}\in \lang{L}_\alg{A}}}$,
where:
\begin{itemize}
	\item The domain $A$ is a subset of the real interval $[0,1]$ such that $1\in A$.
	\item For each $n$-ary connective $\circ\in\lang{L}_\alg{A}$,
        its interpretation $\circ^{\alg{A}}$ in \alg A
        is an $n$-ary operation on $A$ (or an element of $A$ if $n=0$).
	\item For every $x,y \in A$:
\begin{itemize}
\item $x \wedge^\alg{A} y = 1$ iff $x=1$ and $y = 1$.
\item $x \vee^\alg{A} y = 1$ iff $x=1$ or $y = 1$.
\item $x \to^\alg{A} y = 1$ iff $x \leq y$.
\end{itemize}
\end{itemize}
\end{definition}

Note that the realization $\wedge^\alg{A},\vee^\alg{A}$
of the connectives $\wedge,\vee$ in the algebra~\alg A
need not be the minimum and maximum (under the usual order of reals).
Even if this will be the case in typical standard algebras,
the only conditions required of $\wedge^\alg{A},\vee^\alg{A}$ are those of Definition~\ref{def:A},
as they already ensure the validity of all theorems given below.
Restricting the interpretation of $\wedge,\vee$ in \alg A to the lattice operations
would thus impose an unnecessary limitation on the class of admissible logics and on the generality of the results.

\begin{example}\label{ex:standard_algebras}
Let us list several prominent algebras that can be seen as standard algebras in the
sense of Definition~\ref{def:A}:
\begin{itemize}
\item The \emph{two-valued Boolean algebra} $\alg{2}$ in the language $\wedge,\vee,\rightarrow,\neg,\0,\1$
      (where $x\rightarrow y$ is defined as $\neg x\vee y$;
      we will not mention arities of well-known connectives).

\item
The \emph{standard \G-algebra} $[0,1]_\G = \tuple{[0,1], \wedge, \vee, \rightarrow, \0, \1}$ (G~for G\"odel), where $\tuple{[0,1], \wedge, \vee,  \0, \1}$ is the lattice $[0,1]$ with the usual order,
        and $x\rightarrow y =1$ if $x\leq y$ and $x\rightarrow y =y$ otherwise.
        (For G-algebras see, e.g., \cite{Baaz-Preining:Handbook,Behounek-Cintula-Hajek:Handbook}.)

\item The \emph{$(n+1)$-valued \G-algebra} $\alg{G}_n$, i.e., the subalgebra of $[0,1]_\G$ with the domain
    $\{0, \frac1{n}, \dots,\allowbreak \frac{n-1}{n},\allowbreak 1 \}$.
    (See, e.g., \cite{Baaz-Preining:Handbook,Behounek-Cintula-Hajek:Handbook}.)

\item The \emph{standard \MV-algebra} $[0,1]_\mathrmL = \tuple{[0,1], \&, \rightarrow, \wedge, \vee, \0, \1}$,
    where $\tuple{[0,1], \wedge, \vee, \0, \1}$ is the lattice $[0,1]$ with the usual order of reals,
    $x\mathbin{\&}y=\max(x+y-1,0)$, and $x\rightarrow y=\min(1-x+y,1)$.
    In \MV-algebras (and their expansions), it is customary to introduce the defined connectives
    $\neg x=x\rightarrow\0$; $x\oplus y=\neg x\rightarrow y$; and $x\ominus y=x\& \neg y$.
    In (the subalgebras of) the standard \MV-algebra, they are realized as
    $1-x$; $\min(x+y,1)$; and $\max(x-y,0)$ respectively.
    Let us remark that \MV-algebras are often introduced in the language $\oplus,\neg,\0$,
    in which case $\&,{\rightarrow},{\wedge},{\vee},{\1}$ are defined connectives;
    the two definitions are term-wise equivalent.
    (For \MV-algebras see, e.g., \cite{Cignoli-Ottaviano-Mundici:AlgebraicFoundations}.)

\item The \emph{$(n+1)$-valued \MV-algebra} $\mathbfitL_n$ ($\mathbfitL$ for {\L}ukasiewicz),
i.e., the subalgebra of $[0,1]_\mathrmL$ with the domain $\{0, \frac1{n}, \dotsc, \frac{n-1}{n},1 \}$.
(See, e.g., \cite{Cignoli-Ottaviano-Mundici:AlgebraicFoundations}.)

\item If the truth constants for all elements of $\mathbfitL_n$
(i.e., nullary connectives $\overline0,\overline{\frac1n},\dots,\overline{\frac{n-1}n},\overline1$
interpreted by the corresponding elements of~$\mathbfitL_n$)
are added to the language,
we denote the resulting expansion of $\mathbfitL_n$ by~$\mathbfitL_n^c$ (see, e.g.,~\cite{Bou-EGR:Many-ValuedModal}).
Analogously we define the expansion $\alg G_n^c$ of $\alg G_n$ by truth constants.

\item The \emph{standard \prl-algebra} $[0,1]_{\prl}$, which is an expansion of $[0,1]_\mathrmL$ by a binary connective~$\odot$,
    interpreted as the usual algebraic product of reals.
    (\prl-algebras are also known as $\mathrm{PMV}$-algebras; see, e.g., \cite[Sect.~5]{Esteva-Godo-Marchioni:Handbook}.)

\item The \emph{standard \LPi-algebra} $[0,1]_{\LPi}$, which is an expansion of $[0,1]_\prl$ by a binary connective $\rightarrow_\Pi$,
    interpreted as $x\rightarrow_\Pi y = 1$ if $x\leq y$ and $x\rightarrow_\Pi y = \frac{y}{x}$ otherwise.
    (See, e.g., \cite[Sect.~5]{Esteva-Godo-Marchioni:Handbook}.)

\item The expansions of $[0,1]_\mathrmL$, $[0,1]_\G$, $[0,1]_{\prl}$, and $[0,1]_{\LPi}$ with nullary
  operations (i.e., constants) $\bar{r}$ for all $r\in[0,1]\cap \mathbb{Q}$,
  where each constant $\bar r$ is interpreted by the rational number $r$.
  We denote these algebras, respectively,
  as $[0,1]_{\Q\mathrmL}$, $[0,1]_{\Q\G}$, $[0,1]_{\Q\prl}$, and $[0,1]_\LPih$.
  (The traditional symbol for the last mentioned algebra
    is due to the fact that in $[0,1]_\LPih$, all rational constants are definable from the constant for $\frac12$.
    For the expansions of the standard \MV-, G-, \prl-, and \LPi-algebra by rational constants see, e.g., \cite{Esteva-Godo-Marchioni:Handbook}.)

\item The expansion of any standard algebra $\alg{A}$ such that $0\in A$ by the unary operation $\tri$ interpreted as $\tri x = 1$
 if $x=1$ and $\tri x=0$ otherwise; we denote this algebra by $\alg{A}^\tri$.
 (The operation $\tri$ is definable in $\mathbfitL_n$, $\mathbfitL_n^c$, $[0,1]_{\LPi}$, and $[0,1]_{\LPih}$,
 so $\mathbfitL_n^\tri=\mathbfitL_n$ modulo term-wise equivalence,
 and similarly for $\mathbfitL_n^{c\,\tri}$, $[0,1]_{\LPi}^\tri$, and $[0,1]_{\LPih}^\tri$.
 For expansions by $\tri$ see, e.g., \cite[Ch.~2.4]{Hajek:1998} or \cite[Sect.~2.2.1]{Behounek-Cintula-Hajek:Handbook}.)
\end{itemize}
\end{example}

Let us recall several standard notions of the algebraic semantics of many-valued logics.
(For a detailed modern exposition see, e.g., \cite{Cintula-Noguera:Handbook}.)

\begin{definition}
Let $\alg{A}$ be a standard algebra.
An $\alg{A}$-\emph{evaluation} is a mapping $e$ assigning an element of $A$ to each propositional variable.
Every $\alg{A}$-evaluation can be uniquely extended to a mapping from $\fmL{L_\alg A}$ into~$A$,
by setting $e(\circ(\vectn{\f})) = \circ^\alg{A}(e(\f_1), \dots, e(\f_n))$
for each $n$-ary connective ${\circ} \in \lang{L}_\alg{A}$ and formulas $\vectn{\f}$.

An $\lang{L}_\alg{A}$-formula $\f$ is \emph{satisfied} by an $\alg A$-evaluation $e$ if $e(\f) = 1$. A formula $\f$ is $\alg{A}$-\emph{satisfiable} if it is satisfied by some $\alg A$-evaluation.

The \emph{logic of\/ $\alg{A}$} is identified with the consequence relation
$\models_\alg{A}$, defined as follows:
\[
    \Gamma \models_\alg{A} \f \text{ if and only if for each $\alg{A}$-evaluation $e$: if } e[\Gamma] \subseteq \{1\}, \text{ then } e(\f) = 1.
\]
\end{definition}

A trivial, but important observation is that the value of a formula \f\ in an \alg A-evaluation depends only on the variables occurring in~\f.
Let $\vec{v}$ be a sequence $v_1,\dots, v_n$ of pairwise different variables; we shall write $\f(\vectn{v})$, or just $\f(\vec{v})$, to denote that all variables occurring in $\f$ are among those in~$\vec{v}$.
Given a formula $\f(\vectn{v})$ and a sequence of formulas $\vectn{\p}$, we shall write $\f(\vectn{\p})$ to denote the formula where each variable $v_i$ is replaced by the formula $\p_i$.

\begin{definition}\label{def:fA}
Given a formula $\f(\vectn{v})$, we define the mapping $\f^\alg{A}\colon A^{n} \to A$
by setting:
\[
    \f^\alg{A}(\vectn{a}) = e(\f),
\]
where $e$ is any $\alg{A}$-evaluation such that $e(v_i) = a_i$.
\end{definition}

For any given standard algebra $\alg{A}$, it is an interesting question how to describe the class
of all functions $\f^\alg{A}$. The classes of functions expressible in prominent standard algebras are described in Table~\ref{tab:functions}. These delimitations are of a ``piecewise'' character, i.e., based on a decomposition of the corresponding hypercube $[0,1]^n$ into domains;
in particular, each row of the table specifies:
(i)~whether the functions are all those which satisfy the other constraints listed on the row, or just the continuous ones;
(ii)~whether the possible domains are all $\mathsf{Q}$-semialgebraic sets,\footnote{%
    A set $S\subseteq[0,1]^n$ is {\em (linear) $\mathsf{Q}$-semialgebraic} if it is a Boolean combination of sets of the form
    $\{\tuple{\vectn{x}} \in [0,1]^n \mid P(\vectn{x})>0\}$ for (linear) polynomials $P$ with integer coefficients.}
or just the linear ones;
and (iii)~how the functions restricted to these domains are characterized.\footnote{%
    In the table, a \emph{shift} means an absolute coefficient of the polynomial.}
For further details see \cite{McNaughton:FunctionalRep,Montagna-Panti:Adding}
or \cite[Sect.~4.1]{Aguzzoli-Bova-Gerla:Handbook},
where also the (more complicated) result for G\"odel algebras is presented.
Let us furthermore remark that the case of $[0,1]_\prl$ is a long-standing open problem related to the so-called Pierce--Birkhoff conjecture~\cite{Birkhoff-Pierce:Conjecture,Lapenta-Leustean:PierceBirkhoff}.

\begin{table}[t]
\begin{center}
\renewcommand{\arraystretch}{1.4}
\small
\begin{tabular}{|l|c|c|l|}
\hline \textbf{Algebra}         & \textbf{Functions}  & \textbf{Domains} & \textbf{Pieces}
\\ \hline $[0,1]_\mathrmL$          & continuous            & linear    & linear functions with integer coefficients
\\ \hline $[0,1]_\mathrmL^\tri$     & all                   & linear    & linear functions with integer coefficients
\\ \hline $[0,1]_{\Q\mathrmL}$      & continuous            & linear    & linear functions with integer coefficients and a rational shift
\\ \hline $[0,1]_{\Q\mathrmL}^\tri$ & all                   & linear    & linear functions with integer coefficients and a rational shift
\\ \hline $[0,1]_\prl^\tri$         & all                   & all       & polynomials with integer coefficients
\\ \hline $[0,1]_{\Q\prl}^\tri$     & all                   & all       & polynomials with rational coefficients
\\ \hline $[0,1]_\LPih$             & all                   & all       & fractions of polynomials with integer coefficients
\\ \hline $[0,1]_\LPi$              & all                   & all       & functions $f$ expressible in $[0,1]_\LPih$ such that $f[\{0,1\}^n]\subseteq\{0,1\}$
\\ \hline
\end{tabular}
\vspace*{2ex}
\caption{Characterization of functions expressible in prominent standard algebras.}\label{tab:functions}
\end{center}
\end{table}

\subsection{Strategic Games} \label{ssec:strategic_games}

We present some basic notions and results concerning strategic games with finitely many players.
In particular, we review the most fundamental solution concept for such games,
namely that of a \emph{Nash equilibrium}, both for pure  and mixed strategies.
The examples in this subsection are intended to illustrate these concepts
and are taken up in later sections to demonstrate that many well known games can be
represented as logical games (in the sense of Definition~\ref{def:logical_game}). % of Section~\ref{sec:representing}.
We use standard game-theoretic notation and terminology; see, e.g., \cite[Ch.~4--5]{Maschler-Solan-Zamir:GameTheory} or~\cite{Osborne-Rubinstein:CourseGameTheory}.

\begin{definition}\label{def:SG}
A strategic game $\Gcal$ is an ordered triple
\[
\Gcal=\tuple{N, \{S_{i}\mid i\in N\},\{f_{i}\mid i\in N\}},\text{ where:}
\]
\begin{enumerate}
\item $N=\{1,\dotsc,n\}$ is a finite set of \emph{players.}
\item Each $S_{i}\neq \emptyset$ is a \emph{strategy set} of player $i\in N$.
\item Putting $S=S_{1}\times\dotsb \times S_{n}$, each function $f_{i}\colon S\to \Rc$ is called the \emph{payoff function} (or: \emph{utility function}) of player~$i$.
\end{enumerate}
\noindent
Note that we do not restrict the cardinality of the strategy sets $S_i\neq \emptyset$ at this point.
A game $\Gcal$ is called \emph{finite} if each $S_{i}$ is finite.
An ordered $n$-tuple of strategies $\vec{s}=\tuple{s_{1},\dotsc,s_{n}}\in S$ is
called a \emph{strategy profile}.
\end{definition}

Throughout the paper we will adhere to the following conventions:
\begin{itemize}
\item Whenever a symbol is related to a particular player~$i$, then we use the corresponding subscript.
Thus, e.g., the strategies of player~$i$ will typically be denoted by $s_i$, $s'_i$, etc.
If the subscripted symbol is itself a tuple,
then the second index will be written as a superscript
(e.g., $s_i^j$ and $v_i^j$).

\item For every player $i\in N$ and a strategy profile $\vec{s}=\tuple{s_{1},\dotsc,s_{n}}\in S$,
by $\vec{s}_{-i}$ we denote the ordered $(n-1)$-tuple $\tuple{s_{1},\dotsc,s_{i-1},s_{i+1},\dotsc,s_{n}}$.
By $\tuple{s'_i,\vec{s}_{-i}}$ we abbreviate
the strategy profile $\langle s_{1},\allowbreak\dotsc,\allowbreak s_{i-1},\allowbreak s'_i,s_{i+1},\dotsc,s_{n}\rangle$.
The utility of player $i\in N$ under the strategy profile $\tuple{s'_i,\vec{s}_{-i}}$ is written as $f_{i}(s'_i,\vec{s}_{-i})$.
\end{itemize}

The solution concept of a Nash equilibrium captures the idea of stability in the given game.
When all players choose their strategies according to a Nash equilibrium, then neither player can
profit from unilaterally deviating to an alternative strategy.

\begin{definition}
Let $\Gcal$ be a strategic game. A strategy profile $\vec{s}^{*}=\tuple{s^{*}_{1},\dotsc,s^{*}_{n}}\in S$ is a \emph{pure Nash equilibrium} of $\Gcal$ if \[f_{i}(s_i,\vec{s}^{*}_{-i})\leq f_{i}(\vec{s}^{*}),\] for every player $i\in N$ and every strategy $s_i\in S_{i}$.
\end{definition}

Since one of the main aims of the paper is to show that a very wide range of strategic games,
in particular all finite games, can be represented as logical games, we provide
several different examples here that will be taken up again in later sections.
We focus on examples that cannot be directly modeled as either Boolean or
{\L}ukasiewicz games.

\begin{example}[New Technology] \label{ex:adapt}
Suppose that there are three firms sharing a market.
In face of a new technology each firm has to decide whether to adopt it or else to stay put.
We assume that the total value of the market remains unchanged; only the
relative competitiveness of the firms may change in accordance to their decisions.
If only one firm decides to adopt the new technology, it will gain a certain competitive
advantage $c>0$ and each of the other two firms looses $\sfrac{c}{2}$, accordingly.
 If two firms decide to adopt then they split the competitive gain, receiving
$\sfrac{c}{2}$ each, and the third firm has to bear the full loss~$c$. If either
none or all firms adopt the new technology, no firm will gain or loose anything.

The payoff vectors of the resulting 3-player zero-sum game are as follows:
\[
\begin{array}{|cc||c|c|}
\multicolumn{4}{l}{\mbox{{\bf Firm 3:} adopt}} \\
\hline
\hspace*{1.2ex}& & \multicolumn{2}{c|}{\rule[-0mm]{0mm}{4mm}\mbox{{\bf Firm 2}}}  \\
& \hspace*{-4ex}{\raisebox{0ex}[0pt]{{\bf Firm 1}}}
& \mbox{adopt} &  \mbox{stay put} \\
\hline \hline
& \rule[-4mm]{0mm}{10mm}
\mbox{adopt} &   (0,0,0) &  (\sfrac{c}{2},-c,\sfrac{c}{2}) \\
\cline{2-4}
%\raisebox{3.2ex}[0pt]{{\bf Firm 1}}
& \rule[-4mm]{0mm}{10mm}
\mbox{stay put} & (-c,\sfrac{c}{2},\sfrac{c}{2}) & (\sfrac{-c}{2},\sfrac{-c}{2},c) \\
\hline
\end{array}
\quad
\begin{array}{|cc||c|c|}
\multicolumn{4}{l}{\mbox{{\bf Firm 3:} stay put}} \\
\hline
\hspace*{1.2ex}& & \multicolumn{2}{c|}{\rule[-0mm]{0mm}{4mm}\mbox{{\bf Firm 2}}}  \\
& \hspace*{-4ex}{\raisebox{0ex}[0pt]{{\bf Firm 1}}}
& \mbox{adopt} &  \mbox{stay put} \\
\hline \hline
& \rule[-4mm]{0mm}{10mm}
\mbox{adopt} &  (\sfrac{c}{2},\sfrac{c}{2},-c) &  (c,{\sfrac{-c}{2}},{\sfrac{-c}{2}}) \\
\cline{2-4}
%\raisebox{3.2ex}[0pt]{{\bf Firm 1}}
& \rule[-4mm]{0mm}{10mm}
\mbox{stay put} & (\sfrac{-c}{2},c,\sfrac{-c}{2}) & (0,0,0) \\
\hline
\end{array}
\]
It is not hard to see that the only pure Nash equilibrium of the game arises if every firm  adopts
the new technology. (In fact, adopting is a dominating strategy for each firm.)
\end{example}

Examples of games with infinite strategy spaces arise very naturally in many applications;
however, they are highly non-trivial to analyze, in general.

\begin{example}[Vickrey Auction] \label{ex:Vickrey}
As pointed out, e.g., in \cite{Osborne:IntroductionGameTheory,Osborne-Rubinstein:CourseGameTheory}
many types of auctions can be modeled as strategic games under certain assumptions.
Of particular interest is the second-price sealed-bid auction with perfect
information, also called \emph{Vickrey auction,} since, in contrast to more familiar types of auctions,
bidders have an incentive to bid their true value.

Our strategic game representing a Vickrey auction has $n$ players (bidders).
Independently from the others, each player $i\in N=\{1,\dots,n\}$
associates a rational value $p_i\ge0$ 
to the object sold in the auction. This value will be used to define the payoff function and thus is
assumed to be known to all players, as required in all strategic games.
The strategy set $S_i$, i.e.,
the set of possible bids of player~$i$, is identified with the interval $[0,t]$, where
$t$ is some fixed rational maximal bidding value,
greater than~$\max_{i\in N}p_i$.
Each player only knows her own bid $b_i\in [0,t]$. The assumption that values and bids are capped at some fixed $t$ is not restrictive,
as real-life bidders always have finite bankrolls.) 
The object is assigned to the player with the highest bid.
If more than one player has chosen the same highest bid then the object
is assigned to the player with the lowest index among those sharing the same highest bid.
The price to be paid for the object by that player is the highest bid made by
any other player.   This amounts to the following payoff function:\looseness-1
\[
f_i(b_1,\ldots,b_n)=\begin{cases}
	p_i-\max_{j\neq i} b_j & \text{if } i=\min\{j\in N\mid b_j=\max_{k\in N} b_k\} \\
            0 & \text{otherwise.}
	\end{cases}
\]
A Nash equilibrium for this game is $\tuple{b_1,\ldots,b_n}=\tuple{p_1,\ldots,p_n}$, which
means that all players bid their true value. However there are also other
Nash equilibria; for example,
if $p_1>p_2>\dots>p_n$, then
$\tuple{b_1,\ldots,b_n}=\tuple{p_1,0,\ldots,0}$ and
$\tuple{b_1,\ldots,b_n}=\tuple{p_2,p_1,0,\ldots,0}$
are Nash equilibria, too.
\end{example}

\begin{example}[Electoral Competition]\label{ex:EC}
The following simplified model of electoral competition between two candidates appears in \cite{Kroupa-Majer:StrategicReasoning}. Let $N=\{1,2\}$ and $S_1=S_2=[0,1]$. Put:
\[
f_1(s_1,s_2)=\begin{cases}
    0           & \text{if } s_1,s_2\in \{0,1\} \\
    \sfrac 12   & \text{if } s_1=s_2=\sfrac 12\\
    1           & \text{if } s_1=\sfrac 12, s_2\in \{0,1\}
\end{cases}
\]
and define the values $f_1(s_1,s_2)$ elsewhere by a linear interpolation. Put $f_2=1-f_1$.

\begin{figure}
\centering
\includegraphics[scale=0.7]{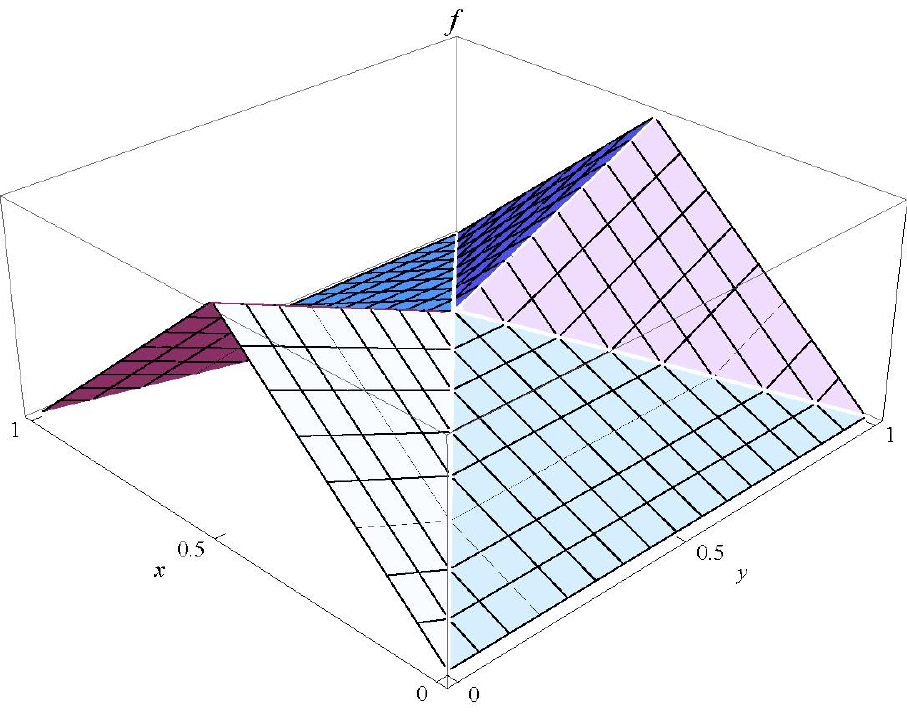}
\caption{\small Payoff function $f_1(s_1,s_2)=f(x,y)$ for player 1 in the simplified electoral competition model.}
\end{figure}

This game captures the following simplified version of Hotelling's electoral competition model \cite{Osborne:IntroductionGameTheory}. Assume that players 1 and 2 are candidates choosing policies $s_1,s_2\in [0,1]$ in order to win the elections. Each citizen has preferences over policies and
votes for either player 1 or 2. In the latest poll, the preferences show that player 1
\begin{itemize}\addtolength\itemsep{-3pt}
\item cannot attract extreme right or extreme left voters at all,
\item is preferred by centrist voters whenever player 2 chooses any of the two extreme policies, and
\item ties with player 2 if both adopt the centrist policy.
\end{itemize}
The values of $f_i$ express the ratio of votes for player $i$ {so that $f_1+f_2=1$}. The unique pure Nash equilibrium in this game is the point $\tuple{\sfrac 12, \sfrac 12}$, representing the simultaneous choice of the centrist policies.
\end{example}

The next example shows that even very simple games may not admit pure Nash equilibria.

\begin{example}[Matching Pennies]\label{ex:MP}
\emph{Matching Pennies} is a game in which  each player $i\in N=\{1,2\}$ secretly selects
one of the sides of a coin.
We may put $S_1=S_2=\{h,t\}$, with $h$ for ``head'' and $t$ for ``tail''.
After their choices are made public, the payoffs
of players 1 and 2 are given by the following table:
\begin{center}
\renewcommand{\arraystretch}{1.3}
\begin{tabular}{|c|c|c|}
	\hline  & $h$ & $t$ \\
	\hline \ $h$\ &\ $(1,-1)$\  &\ $(-1,1)$\ \\
	\hline \ $t$\ &\ $(-1,1)$\ &\ $(1,-1)$\ \\
	\hline
\end{tabular}
\end{center}
It is easy to see the game of \emph{Matching Pennies} has no pure Nash equilibrium.
\end{example}

The previous example motivates the introduction of mixed strategies over the sets $S_i$.
In order to avoid technicalities we confine our attention to mixed strategies in \emph{finite} strategic games (cf.\ Remark~\ref{rem:FinitenessMixed}).

\begin{definition}\label{def:mixed-strategy}
Let $\Gcal=\tuple{N, \{S_{i}\mid i\in N\},\{f_{i}\mid i\in N\}}$ be a finite strategic game. A probability distribution $p_i$ on the strategy set $S_i$ of player $i\in N$ is called a \emph{mixed strategy} of player $i$.
More precisely, $p_i$ is a function $S_i\to [0,1]$ such that $\sum_{s_i\in S_i}p_i(s_i)=1$.
By $\Delta_i$ we denote the set of all mixed strategies of player $i$.

For any \emph{mixed strategy profile} $\vec{p}=\tuple{p_1,\dotsc,p_n}\in \Delta=\Delta_1 \times \dotsb \times \Delta_n$
we set:
\[
    \Ex_i(\vec{p})=\sum_{\vec{s}\in S} \Bigl( f_i(\vec{s}) \cdot\prod_{i\in N}p_i(s_i) \Bigr).
\]
The function $\Ex_i\colon \Delta \to \Rc$ is the \emph{expected payoff (utility)} of player $i\in N$;
its dependence on the payoff function $f_i$ is tacitly understood. A mixed strategy profile $\vec{p}^{*}=\tuple{p^{*}_{1},\dotsc,p^{*}_{n}}\in\Delta$ is a \emph{mixed Nash equilibrium} of $\Gcal$ if \
\begin{equation}\label{def:mixedNE}
    \Ex_{i}(p_i,\vec{p}^{*}_{-i})\leq \Ex_{i}(\vec{p}^{*}),
\end{equation}
for every player $i\in N$ and every mixed strategy $p_i\in \Delta_{i}$.

\end{definition}
\begin{theorem}[Nash]
	Every finite strategic game $\Gcal$ has a mixed Nash equilibrium.
\end{theorem}

Note that the mixed strategy spaces $\Delta_i$ 
are uncountably infinite and each of them contains the original pure strategy space $S_i$
via the embedding $s_i\in S_i\mapsto \delta_{s_i}\in \Delta_i$, where $\delta_{s_i}$ is the
Dirac probability distribution at $s_i$:
\[
\delta_{s_i}(t_i)=\begin{cases}
    1 & \text{if } t_i=s_i\\
    0 & \text{if } t_i\neq s_i
\end{cases}\quad\text{for each } t_i \in S_i.
\]
For every pure strategy $s_i\in S_i$ let $\tuple{s_i,\vec{p}_{-i}}$ denote
the mixed strategy profile in which the mixed strategy of player $i$ is the Dirac distribution $\delta_{s_i}$ concentrated at $s_i$.
It can be shown that condition~\eqref{def:mixedNE} of Definition~\ref{def:mixed-strategy} only needs to be
checked against the pure strategies of each player:

\begin{proposition}[{\cite[Corollary 5.8]{Maschler-Solan-Zamir:GameTheory}}]\label{p:Games-MixedNE}
For any finite strategic game $\Gcal$ the following are equivalent:
\begin{enumerate}
	\item $\vec{p}^*$ is a mixed Nash equilibrium of $\Gcal$.
	\item For every player $i\in N$ and every pure strategy $s_i\in S_i$,
	\[
	\Ex_i(s_i,\vec{p}^*_{-i})\leq \Ex_i(\vec{p}^*).
	\]
\end{enumerate}	
\end{proposition}	
\begin{example}[Love and Hate]\label{ex:love-hate}
	The countable game called \emph{Love and Hate} was analyzed in \cite{Capraroa-Scarsini:EquilibriaAlgebraicApproach}.
	Here we present its finite variant $\mathcal{LH}$.	It is played by an even number $n=2k$  of players. Let $m$ be an even positive integer. Each strategy space $S_i$ is equal to the set $\left\{0,\tfrac{1}{m},\tfrac{2}{m},\dots,\tfrac{m-1}{m},1\right\}$.	
	Let $h(x,y)=2\cdot \min (\abs{x-y},1-\abs{x-y})$, for every $x,y\in S_i$.
	The payoff functions are defined as follows,
	for every $j=1,\dots,k$ and every strategy profile $\tuple{s_1,\dots,s_n}\in S$:
	\[
	f_{2j-1}(s_1,\dots,s_n)=h(s_{2j-1},s_{2j}), \quad f_{2j}(s_1,\dots,s_n)=1-h(s_{2j},s_{2j+1}).
	\]
It can be shown (and later will be verified in Example~\ref{ex:love-hate-mNE})
	that a mixed
	Nash equilibrium in this game is the $n$-tuple of mixed strategies $\tuple{p_1,\dots,p_n}$ defined for $t_1,r_1, t_3,r_3,\dots,t_{n-1},r_{n-1}\in S_1$ by
	\begin{align*}
	p_1(x)=p_2(x)&=\tfrac{1}{2}\cdot(\delta_{t_1}(x)+\delta_{r_1}(x)),\\
	&\dots \\
	p_{n-1}(x)=p_n(x)&=\tfrac{1}{2}\cdot(\delta_{t_{n-1}}(x)+\delta_{r_{n-1}}(x)),
	\end{align*}
	where $\abs{t_{2j-1}-r_{2j-1}}=\tfrac12$ for each $j\le k$.
\end{example}

\begin{remark}\label{rem:FinitenessMixed}
The assumption of finiteness of strategy spaces in $\Gcal$ makes the ensuing theory much more understandable and technically easier. While it is possible to relax this assumption and define the mixed equilibria as general probability measures for games with infinite strategy sets, many additional assumptions are needed and the existence of a Nash equilibrium in mixed strategies is no longer guaranteed in general---see, e.g., \cite{Dresher:GamesStrategy}. Moreover, probability measures over an infinite universe are not
directly amenable to a logical treatment since they do not admit any finite representation,
in general. The problem of determining and computing the mixed strategies over $[0,1]$
in case of games expressible by formulas in {\L}ukasiewicz logic is studied in \cite{Kroupa-Majer:StrategicReasoning}.
\end{remark}	

It is trivial to observe that any bijective re-labeling of the strategies in a finite game preserves pure Nash equilibria.
Moreover, it holds as well for mixed equilibria, as stated by the next lemma.

\begin{lemma}\label{l:renaming}
Let $\Gcal=\tuple{N, \{S_{i}\mid i\in N\},\{f_{i}\mid i\in N\}}$
and $\Gcal'=\langle N, \{S'_{i}\mid i\in N\},\allowbreak\{f'_{i}\mid \nolinebreak{i\in N}\}\rangle$
be finite strategic games such that, for each $i\in N$, there is a bijection $c_i\colon S_i\to S'_i$ and $f'_i(\vec c(\vec{s}))=f_i(\vec{s})$,
where $\vec{c}(\vec{s})=\vec{c}(s_1,\dotsc,s_n)=\tuple{c_1(s_1),\dotsc,c_n(s_n)}$,
for every $\vec{s}\in S_1\times\dots\times S_n$.
Then the following are equivalent for any mixed strategy profile $\vec{p}=\tuple{p_1,\dotsc,p_n}$ in game $\mathcal{G}$:
\begin{enumerate}
\item $\vec{p}$ is a mixed Nash equilibrium in $\mathcal{G}$.
\item $\vec{c}(\vec{p})=\tuple{p_1\circ c_1^{-1},\dotsc,p_n\circ c_n^{-1}}$ is a mixed Nash equilibrium in $\mathcal{G}'$.
\end{enumerate}
\end{lemma}

\newpage

It can be easily shown that pure Nash equilibria are invariant with respect to order-preserving maps:

\begin{lemma}\label{l:monotoneTrans}
Let $\Gcal=\tuple{N, \{S_{i}\mid i\in N\},\{f_{i}\mid i\in N\}}$ be a strategic game. For every player $i\in N$, let $g_i$ be a real non-decreasing function defined on the range of $f_i$. Then every pure equilibrium in $\Gcal$ is also a pure equilibrium in the game $\hat{\Gcal}=\tuple{N,\{S_{i}\mid i\in N\},\{g_i\circ f_{i}\mid i\in N\}}$.
\end{lemma}

Mixed Nash equilibria are invariant with respect to positive affine transformations of the payoff
functions---see \cite[Theorem 5.35]{Maschler-Solan-Zamir:GameTheory}, for example.

\begin{lemma}\label{l:affineTrans}
Let $\Gcal=\tuple{N, \{S_{i}\mid i\in N\},\{f_{i}\mid i\in N\}}$ be a finite strategic game. For every player $i\in N$, let $a_i>0$, $b_i\in \Rc$, and $g_i=a_i f_i + b_i$. Let $\hat{\Gcal}=\tuple{N,\{S_i\mid i\in N \}, \{g_i\mid i\in N\}}$. Then every mixed Nash equilibrium of\/ $\Gcal$ is also a mixed Nash equilibrium of~$\hat{\Gcal}$.
\end{lemma}

\section{Logical games---representing strategic games}\label{sec:representing}

In this section, we will first (in Subsection~\ref{ss:basic})
formally introduce the notion of a logical game, followed by some useful notational conventions.
In Subsection~\ref{sec:RepreExamples} we define and then illustrate by various examples the
concept of representing a given strategic game as a logical game. Subsection~\ref{expressANDrepre}
discusses expressibility issues. Finally, Subsection~\ref{GeneralRepre} provides a series of general propositions that demonstrate how (wide classes of) finite strategic games can be represented as
logical games at various levels of expressiveness of the underlying algebra.

\subsection{Basic definitions} \label{ss:basic}

We introduce a special kind of strategic games---so-called logical $\alg{A}$-games.
The standard algebra~$\alg{A}$ plays two related roles in the definition of such games:
\begin{itemize}
\item Each player `controls' a set of propositional variables: her strategies are assignments of values from $A$ to
those variables.
\item Each payoff function is expressible by a formula in the language $\lang{L}_\alg{A}$ built from variables controlled by the players; thus
each strategy profile provides the full information needed to evaluate any such `payoff formula'
and the possible payoffs are elements of $A$ as well.
\end{itemize}

\begin{definition} \label{def:logical_game}
A \emph{logical $\alg{A}$-game,}
where \alg A is a standard algebra,
is an ordered tuple
\[
    \Gcal=\tuple{N,V, \{V_i\mid i\in N\},\{S_i\mid i\in N\},\{\f_{i}\mid i\in N\}},\text{ where:}
\]
\begin{enumerate}
\item $N=\{1,\dotsc,n\}$ is a finite set of \emph{players.}
\item $V$ is a finite set of propositional variables.

\item $V_{1},\dotsc,V_n$ are sets of propositional variables forming a partition of $V$.
\item $S_{i}\subseteq A^{V_i}$ is the \emph{strategy set} of player $i\in N$;
    we assume that $S_i$ is non-empty for each $i\in N$.

\item The formula $\f_i$ over variables from $V$ in the language $\lang{L}_\alg{A}$ represents
 the \emph{payoff function} of player $i\in N$; i.e., her payoff
 in the strategy profile $\vec s=\tuple{\vectn{\vec s}} \in S =  S_1\times \dots \times S_n$
 is $e(\f_i)$, in any $\alg{A}$-evaluation $e$
 such that $e(v)=\vec s_j(v)$ for each $j\in N$ and $v\in V_j$.
\end{enumerate}
We say that $\Gcal$ is:
\begin{itemize}
\item \emph{Basic} if $V_i$ is a singleton for each $i\in N$.
\item \emph{Finite} if $S_i$ is finite for each $i\in N$.
\item \emph{Full} if $S_i = A^{V_i}$ for each $i\in N$.
\end{itemize}
Finally, we say that an element $a\in A$ is \emph{$\Gcal$-relevant} if $a = \vec{s}_i(v)$ for some $i\in N$, $v\in V_i$, and $\vec{s}_i \in S_i$; we denote the set of all $\Gcal$-relevant elements of the  carrier set $A$ of~$\alg A$ by $\ES{\alg{A}}{\Gcal}$.
\end{definition}

\newpage

Before introducing the notional conventions and commenting on the definition, we present an example of a logical game that is closely
related to the New Technology game of Example~\ref{ex:adapt}:

\begin{example}\label{ex:adapt-QL}
Let $\mathcal{NT}_{\!\mathbfitL_4^c}$ be a 3-player logical $\mathbfitL_4^c$-game such that:
\begin{itemize}
\item $N=\{1,2,3\}$.
\item $V=\{v_1,v_2,v_3\}$.
\item $V_i=\{v_i\}$ for each $i\in N$.
\item $S_i=\{\{\tuple{v_i,a}\}\mid a\in\{0,1\}\}$ for each $i\in N$
    (i.e., the players can only assign the values $0$ or $1$ to the variable they control).
\item The $\lang L_{\mathbfitL_4^c}$-formulas representing payoffs are as follows
    (see Example~\ref{ex:standard_algebras} for the definitions of the connectives in the algebra $\mathbfitL_4^c$):
    \begin{align*}
        \f_1(v_1,v_2,v_3)&=
            \bigl(\overline{\sfrac12}\oplus\bigl(\overline{\sfrac12}\wedge v_1\bigr)\bigr)\ominus
            \bigl(\bigl(\overline{\sfrac14}\wedge v_2\bigr)\oplus\bigl(\overline{\sfrac14}\wedge v_3\bigr)\bigr)
    \\
        \f_2(v_1,v_2,v_3)&=
            \bigl(\overline{\sfrac12}\oplus\bigl(\overline{\sfrac12}\wedge v_2\bigr)\bigr)\ominus
            \bigl(\bigl(\overline{\sfrac14}\wedge v_1\bigr)\oplus\bigl(\overline{\sfrac14}\wedge v_3\bigr)\bigr)
    \\
        \f_3(v_1,v_2,v_3)&=
            \bigl(\overline{\sfrac12}\oplus\bigl(\overline{\sfrac12}\wedge v_3\bigr)\bigr)\ominus
            \bigl(\bigl(\overline{\sfrac14}\wedge v_1\bigr)\oplus\bigl(\overline{\sfrac14}\wedge v_2\bigr)\bigr)
    \end{align*}
\end{itemize}
Clearly, the game is \emph{basic} (as each player $i$ only controls the single variable~$v_i$),
\emph{finite} (as each player $i$ has only two strategies, namely, $v_i\mapsto0$ and $v_i\mapsto1$),
but \emph{not full} (as the strategies do not exhaust $\{0,\sfrac14,\sfrac12,\sfrac34,1\}^{\{v_i\}}$).
Only the elements $0,1$ are $\mathcal{NT}_{\!\mathbfitL_4^c}$-relevant,
as they are the only elements of $\mathbfitL_4^c$ that can be assigned by the players to the variables they control;
thus, $\ES{\mathbfitL_4^c}{\,\mathcal{NT}_{\!\mathbfitL_4^c}}=\{0,1\}$.

The logical game $\mathcal{NT}_{\!\mathbfitL_4^c}$ is a $\mathbfitL_4^c$-representation of the New Technology game of Example~\ref{ex:adapt},
where each player's strategy $v_i\mapsto1$ represents ``adopt'' and $v_i\mapsto0$ the strategy ``stay put''.
The details of the representation (especially how the payoff formulas correspond to the payoff functions of the strategic game)
will be clarified by Definition~\ref{def:represent} and Example~\ref{ex:adapt-representation} below.
\end{example}

Let us now introduce several notational conventions and identifications that will simplify the further presentation of logical games and formulation of results:
\begin{itemize}
\item Recall that we reserve subscripts for the index of the relevant player;
    a second index, if needed, is written as a superscript.

\item For any $i\in N$, let us enumerate the propositional variables in $V_i$ as $V_i = \{v^1_i, \dots, v_i^{|V_i|}\}$.
The tuple $\tuple{v_{i}^1, \dots, v_{i}^{|V_i|}}$ will be denoted by~$\vec{v}_i$.

\item By Definition~\ref{def:logical_game},\label{pageref:conventions} the strategies of player $i\in N$, or the elements of $S_i$,
are (some) \emph{mappings} from the player's set of controlled variables $V_i$ to $A$.
Thanks to the fixed enumeration of $V_i$, player $i$'s strategies $\vec s_i\in S_i$
can be identified with \emph{tuples} $\tuple{s_i^1,\dots,s_i^{|V_i|}}$ of elements of~$A$,
where $s_i^j=\vec s_i(v_i^j)$, for each $i\in N,j\in\{1,\dots,|V_i|\}$.
Most of the time we will view $i$'s strategies as tuples rather than mappings,
but freely switch between both meanings. (Formally, we just identify $A^{|V_i|}$ and $A^{V_i}$.)

\item Since each player's strategies can be regarded as tuples of elements of $A$,
a strategy profile $\vec s = \tuple{\vectn{\vec{s}}}$ can be viewed as
the tuple of tuples $\langle\tuple{s_1^1,\dots,s_1^{|V_1|}},\dots,\tuple{s_n^1,\dots,s_n^{|V_n|}}\rangle$,
which can in turn be identified with the concatenation of the inner tuples:
\[
    \vec s=\tuple{s_1^1,\dots,s_1^{|V_1|},\dots,s_n^1,\dots,s_n^{|V_n|}},
\]
i.e., a $|V|$-tuple of elements of~$A$.
Simultaneously, the strategy profile can be identified with a mapping from $V$ to~$A$,
assigning to each $v_i^j\in V$ the element $s_i^j\in A$.
That is, each strategy profile can also be regarded as an \emph{evaluation}
(a truth-value assignment) of all propositional variables in $V$.
Again, we will freely switch between these representations of a strategy profile.

\item Similarly, the set $V$ can be regarded as a $|V|$-tuple
\[
    \vec v=\tuple{\vec v_1,\dots,\vec v_n}=\tuple{v_1^1,\dots,v_1^{|V_1|},\dots,v_n^1,\dots,v_n^{|V_n|}}.
\]
This will allow us to write $\f_i(\vec{v})$ to signify that the variables occurring in $\f_i$
are among those in $\vec{v}$ and to write $\f_i^\alg{A}(\vec s)$ for the value of $\f_i$
in the evaluation determined by the strategy profile~$\vec s$.

\item Recall from Lemma~\ref{l:renaming} that we write $\vec{c}(\vec s)$ for $\langle c_1(s^1_1),\dotsc, c_1(s^{|V_1|}_1), c_2(s^{1}_2), \dotsc,c_n(s^{|V_n|}_n)\rangle$. Similarly, we write $\vec e(\vec v)$ for $\tuple{e(v_1^1),\dotsc,e(v_n^{|V_n|})}$.
\end{itemize}

\begin{example}\label{ex:adapt-conventions}
Recall the logical $\mathbfitL_4^c$-game $\mathcal{NT}_{\!\mathbfitL_4^c}$ of Example~\ref{ex:adapt-QL}.
By our conventions, the set $V$ can be regarded as the triple $\vec v=\tuple{v^1_1,v^1_2,v^1_3}$,
or simply $\tuple{v_1,v_2,v_3}$ since the game is basic.
Similarly, each strategy $\vec s_i\in S_i$ can be identified with the 1-tuple (i.e., an element) $s_i^1\in[0,1]$;
thus, due to the limited choice of elements by each player and the fact that the game is basic, we can identify $S_i$ with the set $\{0,1\} \subseteq A$.
Each strategic profile $\vec s\in S$ can thus be viewed as a triple $\tuple{s_1,s_2,s_3}\in \{0,1\}^3 \subseteq A^3$, or as an $\mathbfitL_4^c$-evaluation of the propositional variables $v_1,v_2,v_3$ (by values $0$ or~$1$).
\end{example}

Let us now discuss our definition of a logical \alg A-game.
Definition~\ref{def:logical_game} generalizes three classes of formalized models of strategic games appearing in the literature:
the well-known Boolean games in strategic form of Harrenstein et al.~\cite{Harrenstein-HMW:BooleanGames} are full \alg 2-games; the finite \L ukasiewicz games of~\cite{Marchioni-Wooldridge:LukasiewiczGames,Marchioni-Wooldridge:LukasiewiczGamesTOCL} are
full $\mathbfitL_n$- or $\mathbfitL_n^c$-games;{\footnote{
Marchioni and Wooldridge formulate their results for what in our terminology are full $\mathbfitL_n^c$-games. However, they also show that characteristic functions for all the elements of the domains of these algebras can be
expressed by $\mathbfitL_n$-formulas (see Section~\ref{expressANDrepre} and Definition~\ref{def:express} below).
In this sense also full $\mathbfitL_n$-games are covered
in~\cite{Marchioni-Wooldridge:LukasiewiczGames,Marchioni-Wooldridge:LukasiewiczGamesTOCL}.
}
and the infinite \L ukasiewicz games introduced in~\cite{Marchioni-Wooldridge:LukasiewiczGamesTOCL} are full $[0,1]_{\Q\mathrmL}$-games.

 Our approach, however, is versatile enough to encompass also other types of strategic games, which are not directly captured by the two mentioned subclasses.
In particular, note that whenever the payoff functions in the strategic game are of a type included in Table~\ref{tab:functions} in Section~\ref{ssec:logics}, then we can represent the game
as a logical $\alg{A}$-game, where $\alg{A}$ is the corresponding algebra specified in
the first column of the table.

Furthermore note that if a game $\Gcal$ is full, then  $\ES{\alg{A}}{\Gcal} = A$. However, it should be stressed that in many games $A\neq \ES{\alg{A}}{\Gcal}$: i.e., not all elements of $A$ are available for the players as evaluations for the variables they control. In particular a logical \alg A-game $\Gcal$ can be finite even if the algebra \alg A itself is infinite.
While this may look unintuitive at the first glance, it is actually not too different from the classical case, where we could also assume that strategy sets of the players may be coded as \emph{proper} subsets of real numbers. The main role of $\alg{A}$ is to provide a way to express payoffs as formulas.
The flexibility in modeling strategy sets yields two major advantages over the previous approaches:
\begin{itemize}

\item A much wider class of strategic games can be represented in our framework compared to the previous approaches (for details see Section~\ref{sec:RepreExamples}).

\item We can extend the algebra $\alg{A}$ by adding more elements to
its domain or adding more operations to express further properties of games (see Section \ref{sec:mixedNE}).
Moreover, if we keep
the strategy sets and payoff formulas unchanged, the resulting game will remain essentially the same (as the payoff formulas of the original game are also formulas of the larger game and we have only extended the codomain of their evaluations, but not the evaluations themselves).
\end{itemize}

The following proposition formalizes the second claim using the notion of a \emph{subreduct}:
$\alg{A}$ is a~subreduct of $\alg{B}$ if $A\subseteq B$, $\lang{L}_\alg{A} \subseteq \lang{L}_\alg{B}$, and the operations of $\alg{A}$ are the restrictions of those in $\alg{B}$ to~$A$. (If $\lang{L}_\alg{A} = \lang{L}_\alg{B}$, we speak of a \emph{subalgebra} $\alg{A}$ of $\alg{B}$; if $A= B$, then $\alg{A}$ is called a \emph{fragment} of~$\alg{B}$.)

\pagebreak

\begin{proposition}\label{p:Subreducts1}
Let $\alg{A}$ be a subreduct of $\alg{B}$. Then every logical $\alg{A}$-game is also a logical $\alg{B}$-game.
\end{proposition}

Proposition~\ref{p:Subreducts1} allows us to view all $\alg{A}$-games, for all algebras $\alg{A}$ from Example~\ref{ex:standard_algebras}, as logical $[0,1]_\LPih$-games and use all the expressive power of this logic (see Table~\ref{tab:functions}) to analyze these games (see Section~\ref{sec:mixedNE}, where it will allow us to expressed mixed Nash equilibria of all finite logical games).

The following trivial proposition shows that restricting the set of strategies in \alg A-games preserves pure Nash equilibria, provided they remain available in the restricted game.

\begin{proposition}\label{p:Subreducts2}
Let $\Gcal=\tuple{N,V, \{V_i\mid i\in N\},\{S_i\mid i\in N\},\{\f_{i}\mid i\in N\}}$
and $\Gcal'=\langle N,V, \{V_i\mid {i\in N}\},\{S'_i\mid i\in N\},\{\f_{i}\mid i\in N\}\rangle$
be logical $\alg{A}$-games such that $S_i \subseteq S'_i$ for all $i\in N$.
If the strategy profile $\vec s\in S_1\times\dots\times S_n$ in~\Gcal\ is a pure Nash equilibrium of~$\Gcal'$,
then it is a pure Nash equilibrium of $\Gcal$ as well.
\end{proposition}

Sometimes we can reverse the implication:

\begin{example}\label{ex:LukRationaVSRealGames}
Let $\Gcal=\tuple{N,V, \{V_i\mid i\in N\},\{S_i\mid i\in N\},\{\f_{i}\mid i\in N\}}$
and $\Gcal'=\langle N,V,\allowbreak\{V_i\mid{i\in N}\},\allowbreak\{S'_i\mid{i\in N}\},\allowbreak\{\f_{i}\mid{i\in N}\}\rangle$, where $V = \{\vectn{v}\}$ and for each $i\leq n$: $V_i = V'_i = \{v_i\}$ (i.e., both $\Gcal$ and $\Gcal'$ are basic), $S_i = [0,1]\cap\Q$, and $S'_i = [0,1]$;
i.e., each player $i$ selects the value for the variable $v_i$ from $[0,1]$ in $\Gcal'$ and from $[0,1]\cap\Q$ in \Gcal\ (thus $\Gcal'$ is full, while $\Gcal$ is not). Then every pure Nash equilibrium of $\Gcal$ is a pure Nash equilibrium of $\Gcal'$ as well: see \cite[Prop.~3.6]{Kroupa-Majer:StrategicReasoning}.
\end{example}

\subsection{Representing strategic games---examples}\label{sec:RepreExamples}

As indicated above, logical $\alg{A}$-games can be viewed as special strategic games.
Thus the notions of pure and mixed Nash equilibria are defined for logical games exactly for strategic games, using the payoff functions $\f_i^{\alg A}(\vec s)$.
We have also seen, at the end of Section~\ref{ssec:strategic_games}, that strategic games related to one another by simple transformations share (pure and/or mixed) Nash equilibria.
The following definition spells out what it means to represent a given strategic game by a logical \alg A-game and adapts the corresponding classes of transformations between games to our setting.

\begin{definition}\label{def:represent}
Let $\Gcal=\tuple{N, \{S_{i}\mid i\in N\},\{f_{i}\mid i\in N\}}$ be a strategic game and
let $\alg{A}$ be a standard algebra. We say that $\Gcal$ is
\emph{represented} by a logical $\alg{A}$-game
$\hat{\mathcal{G}}=\langle N,V,\{V_{i}\mid {i\in N}\}, \{S'_{i}\mid i\in N\},\{\f_{i}\mid i\in N\}\rangle$
via $g$ and $\vec c=\tuple{c_i}_{i\in N}$ if:
\begin{enumerate}
\item $g\colon [0,1] \to \Rc$ is a strictly increasing function.
\item $c_i\colon S_i \to S'_i$ is a bijection for each $i\in N$.
\item $f_i(\vec{s}) = g(\phi^\alg{A}_i(\vec c(\vec s)))$ for each $\vec{s}\in S_1\times \dots \times S_n$.
\end{enumerate}
The representation is called \emph{affine} if $g$ is an affine function.
\end{definition}

Using Lemmas~\ref{l:renaming}--\ref{l:affineTrans} we obtain the following basic fact, which shows how the Nash equilibria of a strategic game and its logical representation are related. For a mixed strategy $\vec{p}$ in $\Gcal$ and bijections $c_i$ as above, $\vec{c}(\vec{p})$ denotes the corresponding image of $\vec{p}$ in $\hat{\Gcal}$, just as in Lemma~\ref{l:renaming}.

\begin{lemma}\label{t:Representation}
Let $\Gcal =\tuple{N, \{S_{i}\mid i\in N\},\{f_{i}\mid i\in N\}}$ be a strategic game and let $\hat{\Gcal} = \langle N,V,\{V_{i}\mid {i\in N}\}, \{S'_{i}\mid {i\in N}\},\{\f_{i}\mid {i\in N}\}\rangle$ be a logical $\alg{A}$-game representing $\Gcal$
via $g$ and $\vec c$.
Then:
\begin{enumerate}
\item A strategy profile $\vec{s}^*$ is a pure Nash equilibrium in $\Gcal$ iff $\vec{c}(\vec{s}^*)$ is a pure Nash equilibrium in~$\hat{\Gcal}$.
\item  If $\Gcal$ is finite and the representation is affine, then $\vec{p}^*$ is a mixed Nash equilibrium in $\Gcal$ iff\/ $\vec{c}(\vec{p}^*)$ is a mixed Nash equilibrium in~$\hat{\Gcal}$.
\end{enumerate}
\end{lemma}

In Section~\ref{GeneralRepre} we state a series of propositions, demonstrating that wide classes of finite games can be (affinely) represented by appropriate logical games of various levels of complexity. Ultimately, Proposition~\ref{prop:vii} shows that for a sufficiently expressive algebra $\alg{A}$, \emph{all} finite games can be represented (though not necessarily affinely) as logical $\alg{A}$-games.
We will see that many of those representations are nevertheless affine and thus preserve even mixed equilibria.
However, the representations that are directly obtained from those propositions are using rather complex formulas, in general.
In the following, we revisit the examples from Section~\ref{ssec:strategic_games}
to indicate that many prominent strategic games can in fact be logically represented in a more compact and natural fashion.

\begin{example}\label{ex:adapt-representation}
Recall the strategic game \emph{New Technology} of Example~\ref{ex:adapt} (let us denote it by~$\mathcal{NT}$)
and the logical $\mathbfitL_4^c$-game $\mathcal{NT}_{\!\mathbfitL_4^c}$ of Example~\ref{ex:adapt-QL}.
Definition~\ref{def:represent} makes precise in which sense $\mathcal{NT}_{\!\mathbfitL_4^c}$ represents $\mathcal{NT}$:
\begin{itemize}
\item There are just two strategies (``adopt'' and ``stay put'') for each player $i$ in $\mathcal{NT}$.
    As hinted in Example~\ref{ex:adapt-QL}, these are encoded by assigning the values $1$ and $0$,
    respectively, to the variable $v_i$ the player controls in $\mathcal{NT}_{\!\mathbfitL_4^c}$.
    This provides the bijection $c_i$ between the two-element strategy sets of player~$i$ in both games.

\item Recall from Example~\ref{ex:adapt-conventions} that strategy profiles in $\mathcal{NT}_{\!\mathbfitL_4^c}$
    can be viewed as triples $\tuple{a_1,a_2,a_3}\in\{0,1\}^3$ evaluating the variables $v_1,v_2,v_3$.
    By the interpretation of the connectives in $\mathbfitL_4^c$ (see Example~\ref{ex:standard_algebras}),
    the payoff formula $\f_1$ of $\mathcal{NT}_{\!\mathbfitL_4^c}$ (defined in Example~\ref{ex:adapt-QL}) evaluates as
    \[
        \f_1^{\mathbfitL_4^c}(a_1,a_2,a_3) = \biggl(\frac12+\frac{a_1}2\biggr)-\biggl(\frac{a_2}4+\frac{a_3}4\biggr)
    \]
    for each $a_1,a_2,a_3\in\{0,1\}$, and similarly for $\f_2$ and $\f_3$.
    Observe that
    the resulting values $\{0,\frac14,\frac12,\frac34,1\}$ of the payoff formulas $\f_i$ in $\mathbfitL_4^c$ can be transformed to the corresponding payoff values
    \[
        f_i(\vec s)\in\{-c,\sfrac{-c}2,0,\sfrac c2,c\}
    \]
    (as defined for each corresponding strategy profile $\vec s$ of $\mathcal{NT}$ in Example~\ref{ex:adapt})
    by the strictly increasing function $g\colon x\mapsto 2c(x-\frac12)$ from $[0,1]$ to~$\Rc$.
    The representation is affine, as $g$ is clearly an affine function.
\end{itemize}
Thus the logical game $\mathcal{NT}_{\!\mathbfitL_4^c}$
represents the finite strategic game \emph{New Technology} affinely,
therefore both games have the same pure and mixed Nash equilibria
(modulo the transformation by $\vec c$ and~$g$).
\end{example}

\begin{example}\label{ex:love-hate-representation}
The finite variant $\mathcal{LH}$ of \emph{Love and Hate} from Example~\ref{ex:love-hate} can be affinely represented as a (full and basic) logical $\mathbfitL_m$-game $\mathcal{LH}_{\mathbfitL_m}$.
First observe that each strategy space $S_i=\left\{0,\tfrac{1}{m},\tfrac{2}{m},\dots,\tfrac{m-1}{m},1\right\}$
is just the universe of the finite {\L}ukasiewicz chain $\mathbfitL_m$.
We can thus take $V_i=\{v_i\}$ and $c_i$ as identity for each $i\in N$.

Using the interpretation of connectives in $\mathbfitL_m$ (see Example~\ref{ex:standard_algebras}),
one can easily verify that $h(x,y)=\eta^{\mathbfitL_m}(x,y)$ for the $\lang L_{\mathbfitL_m}$-formula $\eta$ defined as
	\[
	\eta(x,y)=\bigl(\theta(x,y)\wedge \neg\theta(x,y)\bigr)\oplus\bigl(\theta(x,y)\wedge \neg\theta(x,y)\bigr),
	\]
where $\theta(x,y)=\neg (x\to y) \vee \neg (y\to x)$.
(Observe that $\theta^{\mathbfitL_m}(x,y)=\abs{x-y}$ for every $x,y\in S_i$.)
The payoff functions are thus directly expressible by $\lang L_{\mathbfitL_m}$-formulas $\f_{2j-1}=\eta(v_{2i-1},v_{2i})$ for odd players
and $\f_{2j}=\neg\eta(v_{2j},v_{2j+1})$ for even players.
The representation is affine, as the transformation $g$ is the identity function.
\end{example}

\begin{example}\label{ex:Vickrey-QLD}
Recall the \emph{Vickrey Auction} game from Example \ref{ex:Vickrey}.
This game of $n$ players is determined by values each player associates to the object sold in the auction. Recall that we assume that
the values \vectn p
are non-negative rational numbers and that bids \vectn b of all players are from the interval $[0,t]$ for some rational number $t > \max_{i\le n}p_i$.

Let us consider a $[0,1]^\tri_{\Q\mathrmL}$-game $\mathcal{V\!A}_{\Q\mathrmL}^{\tri}$,
where $N=\{1,2, \dots, n\}$, $V=\{\vectn{v}\}$, and for each $i\in N$, we put $V_i=\{v_i\}$, $S_i=[\frac12,1]$,
and the payoff formulas $\f_i(\vectn v)$ are constructed as follows.

Let $\overline{r_i}=\overline{\textstyle\frac{t+p_i}{2t}}$
be the truth constant corresponding to $\frac{t+p_i}{2t}$
(recall that we assume $t$ and all $p_i$ to be rational).
Let $\vec v=\tuple{\vectn v}$
and define:
\begin{align*}
    \kappa_i(\vec v)&=\bigvee_{j\ne i}v_j
\\  \iota_i(\vec v)&=\tri\biggl(\Bigl(\bigvee_j v_j\Bigr)\rightarrow v_i\biggr)\wedge
        \neg\tri\Bigl(v_i\rightarrow\bigvee_{j<i}v_j\Bigr)
\end{align*}
(if $i=1$, then the empty disjunction $\bigvee_{j<i}v_j$ is understood as~$0$).
The payoff formula $\f_i$ is defined as follows
(for the definition of connectives in the algebra $[0,1]^\tri_{\Q\mathrmL}$ see Example~\ref{ex:standard_algebras}):
\[
    \f_i(\vec v)=
        \biggl(\overline{\textstyle\frac12}\oplus
            \Bigl(\iota_i(\vec v)\wedge\bigl(\overline{r_i}\ominus\kappa_i(\vec v)\bigr)\Bigr)\biggr)\ominus
        \Bigl(\iota_i(\vec v)\wedge\bigl(\kappa_i(\vec v)\ominus\overline{r_i}\bigr)\Bigr).
\]
We can show that $\mathcal{V\!A}_{\Q\mathrmL}^{\tri}$ affinely represents the Vickrey auction game via
$g(x) = 2t(x-\frac12)$ and $c_i(x) = \frac{t+x}{2t}$ for each $i\in N$.

First observe that $\iota_i^{[0,1]^\tri_{\Q\mathrmL}}(\vec v)$ indicates the player who wins the auction.
Indeed, for each $\vectn{b} \in [0,t]$,
\[
\iota_{i}^{[0,1]^\tri_{\Q\mathrmL}}(\vec c(\vec b))
=\begin{cases}
	1 & \text{if $\displaystyle b_i\ge\max_{j\in N}b_j$ and $\displaystyle b_i>\max_{j<i}b_j$}, \\
            0 & \text{otherwise,}
	\end{cases}
\]
where $\vec c(\vec b)$ denotes the tuple $\tuple{c_1(b_1),\dots,c_n(b_n)}$.
Furthermore notice that if $i$ wins the auction, then $\kappa_i$ represents the second highest bid,
i.e., the price to be paid for the auctioned object:
\[
\kappa_{i}^{[0,1]^\tri_{\Q\mathrmL}}(\vec c(\vec b))=
    \max_{j\ne i}c_j(b_j)=\textstyle\frac12+\frac{\max_{j\ne i}b_j}{2t}.
\]
Consequently, by the semantics of the connectives in $[0,1]^\tri_{\Q\mathrmL}$
(see Example~\ref{ex:standard_algebras}), for each $\vectn{b} \in [0,t]$:
\[
\f_{i}^{[0,1]^\tri_{\Q\mathrmL}}(\vec c(\vec b))
=\begin{cases}
	\frac12+\frac{p_i-\max_{j\ne i}b_j}{2t} & \text{if } i=\min\{j\in N\mid b_j=\displaystyle\max_{k\in N} b_k\}; \\
            \frac12 & \text{otherwise.}
	\end{cases}
\]
Thus indeed, $f_i(b_1,\ldots,b_n) = g\bigl(\f_i^{[0,1]^\tri_{\Q\mathrmL}}(\vec c(\vec b))\bigl)$.
\end{example}

\begin{remark}
Marchioni and Wooldridge  \cite{Marchioni-Wooldridge:LukasiewiczGamesTOCL} aim to show that
(a variant of) the second-price
sealed-bid auction with perfect information can be formalized as a {\L}ukasiewicz game.
However, to achieve that goal they explicitly impose some highly problematic restrictions:
(1) all players are assumed to assign the same value to the object in question;
(2) the players can only submit bids that are smaller or equal to the assigned value;
(3) the payoff is set to $0$ for all players if there is a tie at the highest bid.
Since the value assigned to the object is common knowledge (as required in any strategic game),
each of these assumptions trivializes the game.
Jointly these restrictions amount to a game that hardly reflects any essential
feature of the Vickrey auction. In any case, our logical game $\mathcal{V\!A}_{\Q\mathrmL}^{\tri}$
does not impose any of the three mentioned restrictions.
\end{remark}

\begin{example}\label{ex:EC-L}
Recall the \emph{Electoral Competition} model introduced in Example \ref{ex:EC}. It follows directly from the definitions of the two payoffs functions $f_1$ and $f_2$ that both are continuous and piecewise linear (affine), where each of the finitely many linear pieces has only integer coefficients. We have seen in Table~\ref{t:Representation}
that these functions are represented two-variable formulas in infinite-valued \L ukasiewicz logic (see \cite{McNaughton:FunctionalRep,Cignoli-Ottaviano-Mundici:AlgebraicFoundations} for details).
This means that the electoral model is representable as a logical $\alg{A}$-game $\mathcal{EC}_{\mathrmL}$,
where $\alg{A}$  is the standard \MV-algebra (see Example~\ref{ex:standard_algebras}).
\end{example}

\begin{example}\label{ex:MP-2}
Recall the game of \emph{Matching Pennies} from Example \ref{ex:MP}. This game is represented by the Boolean game $\mathcal{MP}_{\!\alg 2}$ specified in Table~\ref{tab:MP}, resulting from the original payoff table by
applying the affine transformation $x\mapsto \tfrac 12 x + \tfrac 12$. Thus, $A=\{0,1\}$ and
we may identify $h$ with $0$ and $t$ with $1$.
The payoff of player~$1$ is then determined by the Boolean function $f_1(v_1,v_2)=1-\abs{v_1-v_2}$, expressible by the classical propositional formula
$\f_1 = (v_1 \wedge \neg v_2) \vee (\neg v_1 \wedge v_2)$;
$f_2$ can be represented either analogously or simply as $\neg \f_1$.
	\begin{table}
		\begin{center}
    \renewcommand{\arraystretch}{1.4}
	\begin{tabular}{|c|c|c|}
		\hline  & $h$ & $t$ \\
		\hline \ $h$\ &\ $(1,0)$\  &\ $(0,1)$\ \\
		\hline \ $t$\ &\ $(0,1)$\ &\ $(1,0)$\ \\
		\hline
	\end{tabular}	
\vspace*{1ex}
\caption{Matching Pennies with transformed payoffs}	
\label{tab:MP}
\end{center}
\end{table}
\end{example}

\subsection{Expressible games}\label{expressANDrepre}

In Sections~\ref{sec:pureNE} and~\ref{sec:mixedNE} we will study the expressibility of
pure and mixed equilibria, respectively, by formulas of the logics in question.
For that purpose we will need an additional condition on logical games.

\begin{definition}\label{def:express}
Let $\alg{A}$ be a standard algebra and $a\in A$.
We say that $\alg{A}$ has:
\begin{itemize}
\item a \emph{(definable) truth constant} for $a$ if there is an $\lang{L}_{\alg{A}}$-formula $\bar a$ such that $e(\bar a) = a$ for every $\alg{A}$-evaluation~$e$;

\item a \emph{pseudo-characteristic formula} for $a$ if there is an $\lang{L}_{\alg{A}}$-formula $\x_a$ over a single variable such that for every $x\in A$ we have $\x^\alg{A}_a(x) = 1$ iff $\nolinebreak{x=a}$;

\item a \emph{characteristic formula} for $a$ if there is an $\lang{L}_{\alg{A}}$-formula $\delta_a$ over a single variable such that for every $x\in A$ we have $\delta^\alg{A}_a(x) = 1$ if $x = a$ and $\delta^\alg{A}_a(x)=0$ otherwise.
\end{itemize}
We say that a logical $\alg{A}$-game $\Gcal$ is:
\begin{itemize}
\item \emph{weakly expressible} if there is a pseudo-characteristic formula for each $a \in \ES{\alg{A}}{\Gcal}$ in $\alg{A}$;
\item \emph{expressible} if there is a truth constant for each $a \in \ES{\alg{A}}{\Gcal}$ in~$\alg{A}$.
\end{itemize}
\end{definition}

Clearly if there is a truth constant $\bar a$ for $a$ in~\alg A, then
$\x_a(p) = (\nolinebreak{p \rightarrow \bar a})\wedge(\nolinebreak{\bar a \rightarrow p})$
is a pseudo-characteristic formula for $a$. Thus all expressible games are weakly expressible.
Also note that if $\lang{L}_{\alg{A}}$ contains the connective $\tri$ (see Example~\ref{ex:standard_algebras}) and $\x_a$ is a pseudo-characteristic formula for $a$ in~\alg A, then $\tri\x_a$ is a characteristic formula for~$a$.

Since the Boolean algebra \alg 2 has truth constants for both of its elements $0$ and~$1$, all Boolean games are expressible. Note that this includes, e.g., the game $\mathcal{MP}_{\!\alg 2}$ of Example~\ref{ex:MP-2}. The analogous claim is no longer true for (finite-valued) \L ukasiewicz games, but well-known results tell us the following:
\begin{lemma}\label{l:Luk-folklore}
~
\begin{enumerate}
\item\label{l:Luk-folklore-prime}
Let $a=\frac mn$, where $m$ and $n$ are relatively prime positive integers with $m\leq n$.  Then there is an $\lang L_{[0,1]_\mathrmL}$-formula $\xi_{m,n}(v)$ such that $\xi_{m,n}^{[0,1]_\mathrmL}(a)=\frac1n$
    and $\xi_{m,n}^{[0,1]_\mathrmL}(x)<\frac1n$ if $x\ne a$.

\item\label{l:Luk-folklore-pseudochar}
    For each rational $a\in [0,1]$ there is a pseudo-characteristic $\lang L_{[0,1]_\mathrmL}$-formula.
\item\label{l:Luk-folklore-rational-no}
    There is neither a characteristic formula nor a definable truth constant in $[0,1]_\mathrmL$ for any rational $a\notin\{0,1\}$.
\item\label{l:Luk-folklore-irrational}
    There is no pseudo-characteristic formula (so, \emph{a fortiori,} no characteristic formula nor a definable truth constant) in $[0,1]_\mathrmL$ for any irrational $a\in[0,1]$.
\item\label{l:Luk-folklore-finite}
Let $n$ be a positive integer and $m \leq n$. Then there is a characteristic formula\label{pageref:char_formula_Ln} for $\frac{m}{n}$ in $\mathbfitL_n$.
\end{enumerate}
\end{lemma}

\begin{proof}
Claim~\ref{l:Luk-folklore-prime} follows, e.g., from the theory of Farey--Schauder hats as developed in  \cite[Section~3]{Cignoli-Ottaviano-Mundici:AlgebraicFoundations}. To obtain claim~\ref{l:Luk-folklore-pseudochar} for any rational $a=\frac mn$ (where $m,n$ are relatively prime), it is sufficient to take the formula $\bigoplus_{i=1}^n\xi_{m,n}$, where $\xi_{m,n}$ is as in  claim~\ref{l:Luk-folklore-prime}.
Claims~\ref{l:Luk-folklore-rational-no} and~\ref{l:Luk-folklore-irrational} are direct consequences of functional representation of \L ukasiewicz infinite-valued logic (see Table~\ref{tab:functions}).
Finally, in order to obtain claim~\ref{l:Luk-folklore-finite}
it suffices to take the formula
$\bigop{\&}\nolimits_{i=1}^n\x_{\frac{m}{n}}$,
where $\x_{\frac{m}{n}}$ is a pseudo-characteristic formula for ${\frac{m}{n}}$ in $[0,1]_\mathrmL$ (a more involved proof of this fact is provided in \cite[Lemma~3]{Marchioni-Wooldridge:LukasiewiczGames} and in~\cite[Lemma~7.2]{Marchioni-Wooldridge:LukasiewiczGamesTOCL}).
\end{proof}

Thus all logical $\mathbfitL_n$-games (including, e.g., the game $\mathcal{LH}_{\mathbfitL_m}$ of Example~\ref{ex:love-hate-representation})
are weakly expressible and so are all logical $[0,1]_\mathrmL$-games $\Gcal$ such and only such that $\ES{[0,1]_{\mathrmL}}{\Gcal} \subseteq \Q$. Observe that, consequently, the game $\mathcal{EC}_{\mathrmL}$ of Example~\ref{ex:EC-L} is not even weakly expressible (although it would become weakly expressible if restricted to rational payoffs).

Furthermore, recall that the algebras $\mathbfitL_n^c$ and $\alg{G}_n^c$ introduced in Example~\ref{ex:standard_algebras} contain truth constants for all elements of their domains;
consequently, all $\mathbfitL_n^c$- and $\alg{G}_n^c$-games are expressible
(including, e.g., the game $\mathcal{NT}_{\!\mathbfitL_4^c}$ of Example~\ref{ex:adapt-QL}).
Example~\ref{ex:standard_algebras} also introduced several standard algebras containing truth constants $\bar{r}$ for all $r\in[0,1]\cap \mathbb{Q}$: clearly for any such algebra $\alg{A}$ and any \alg A-game $\mathcal G$, if $\ES{\alg{A}}{\Gcal}\subseteq \mathbb{Q}$ then $\Gcal$ is expressible. (Thus, e.g., the game $\mathcal{V\!A}_{\Q\mathrmL}^{\tri}$ of Example~\ref{ex:Vickrey-QLD} is expressible.)

Here we can also illustrate the usefulness of Proposition~\ref{p:Subreducts1}:
Consider a $[0,1]_{\mathrmL}$-game $\Gcal$ such that
$\{0,1\} \subsetneq \ES{[0,1]_{\mathrmL}}{\Gcal} 
\subseteq \Q$. We know that $\Gcal$ is weakly expressible, but not expressible. However, by Proposition~\ref{p:Subreducts1}, we can see $\Gcal$ as a $[0,1]_{\Q\mathrmL}$-game which has the same pure and mixed equilibria and is clearly expressible.

\subsection{Representing strategic games---general results}\label{GeneralRepre}

In the next three propositions we investigate a particularly simple type of games: finite strategic games
with at most two possible payoff values. Note that all finite
win/loose games  fall into this category.
It is easy to show that all such games (modulo certain encodings of strategies)  can be presented as Boolean games (logical $\alg{2}$-games in our terminology). The only possible complication
arises from the fact that the sets of strategies need not be limited to a binary choice.
This implies that each player may have to control more than one variable, in general,
so that  all  of her possible choices can be represented.
This result is implicit already in~\cite{Harrenstein-HMW:BooleanGames};
however, we make it explicit and prove it, as it provides a convenient
preparation for more complex cases elaborated later.

\begin{convention}\label{conv:onStrategies}
By Lemma~\ref{l:renaming} one may identify  without loss of generality any given set of strategies of a finite game with an initial segment of the set of natural numbers. For sake of conciseness we will do so in the rest of the paper.
\end{convention}

\begin{proposition} \label{prop:ab_i}
Let $\Gcal =\tuple{N, \{S_{i}\mid i\in N\},\{f_{i}\mid i\in N\}}$ be a finite
strategic game such that the union of ranges of all $f_i$ is $\{a,b\}$,
where $a < b$.
For any $i\in N$, let $n_i$ denote the least natural number such that $|S_i| \leq 2^{n_i}$.
Moreover, for any $s_i\in S_i$, let $\tuple{s_i^1, \dots, s_i^{n_i}}\in  2^{n_i}$
denote the binary representation of $s_i$. (Recall that by Convention~\ref{conv:onStrategies} we assume that $s_i$ is a natural number).

Then the game $\Gcal$ is affinely represented by an expressible logical $\alg{2}$-game
$\hat{\Gcal}=\langle N,V,\allowbreak\{V_i\mid {i\in N}\},\allowbreak\{S'_{i}\mid{i\in N}\},\{\f_{i}\mid{i\in N}\}\rangle$
via $g$ and $\vec c=\tuple{c_i}_{i\in N}$, where for each $i\in N$:
\begin{itemize}
\item $V_i = \{v^1_i,\dots, v_i^{n_i}\}$ (and $V=\bigcup_{i\in N}V_i$).
\item $S'_i = \{\tuple{s_i^1, \dots, s_i^{n_i}} \mid s_i\in S_i\}$.
    (Recall that by our conventions, strategies can be identified with tuples of elements of the algebra; see beginning of Section~\ref{ss:basic}.)
\item $
    \displaystyle\f_i =  \bigvee\limits_{\substack{\vec{s}\in S_1\times\dots\times S_n\\f_i(\vec{s}) = b}} \bigwedge\limits_{k \in N} \bigwedge\limits_{j \leq n_{k}}\delta_{s^j_{k}}(v_{k}^j)
    $,

   where $\delta_x(v)$ is a characteristic formula of\/ $x\in\{0,1\}$ in \alg 2
    (e.g., $\delta_0(v)=\neg v$ and $\delta_1(v)=v$).

\item $c_i(s_i) = \tuple{s_i^1, \dots, s_i^{n_{i}}}$ for each $s_i\in S_i$.
\item $g(x) = (b-a)x + a$.
\end{itemize}
\end{proposition}

\begin{proof}
First observe that $\hat{\Gcal}$, as specified above, is indeed a logical $\alg{2}$-game.
Moreover, $\ES{\alg{2}}{\hat{\Gcal}} = \{0,1\}$,
therefore $\hat{\Gcal}$ is expressible (recall that the requisite truth constants $\0$ and $\1$
are part of $\lang{L}_\alg{2}$).
Clearly each $c_i$ is bijective and $g$ is an affine monotone injection.
It remains to check that for each strategy profile~$\vec{s}$,
$$
    f_i(\vec{s}) = g(\phi^{\alg{2}}_i(\vec c(\vec s))).
$$
It is easy to see that for any strategy profile $\vec{s}$ we have:
$$
    e\Bigl(\bigwedge\limits_{k \in N} \bigwedge\limits_{j \leq n_{k}}\delta_{s^j_{k}}(v_{k}^j)\Bigr) =
\begin{cases}
    1 & \text{if } \tuple{e(v^1_k),\dots, e(v^{n_k}_k)} = \tuple{s^1_k,\dots, s^{n_k}_k} = c_k(s_k)\\
        &\mbox{}\quad\hfill \text{for each $k\in N$ (i.e., $e(\vec{v}) = \vec{c}(\vec{s})$);} \\
    0 & \text{otherwise.}
\end{cases}
$$
Therefore $e(\f_i(\vec{v})) =1 = \f^{\alg{2}}_i(\vec{c}(\vec{s}))$ iff  $f_i(\vec{s}) = b$.
\end{proof}

Next we show that, instead of working with $\sum_{i\in N} n_i$ `binary' variables,
 we could represent such games with just one variable for each player,
but at the price of using a logic with more truth values.

\begin{proposition} \label{prop:ab_ii}
Let $\Gcal =\tuple{N, \{S_{i}\mid i\in N\},\{f_{i}\mid i\in N\}}$ be a finite
strategic game such that the union of ranges of all $f_i$ is $\{a,b\}$,
where $a < b$, and let $m= \max_{i\in N}|S_i| -1$.
Then $\Gcal$ is affinely represented by a \emph{basic} weakly expressible logical $\mathbfitL_m$-game
$\hat{\Gcal} = \langle N,V,\{V_{i}\mid{i\in N}\},\allowbreak\{S'_{i}\mid{i\in N}\},\allowbreak\{\f_{i}\mid{i\in N}\}\rangle$
via $g$ and $\vec c=\tuple{c_i}_{i\in N}$, where for each $i\in N$:
\begin{itemize}
\item $V_i = \{v_i\}$ (thus $V=\{v_1,\dots,v_n\}$).
\item $S'_i = \bigl\{\frac{s_i}{m}\mid s_i\in S_i\bigr\}$.
\item $\displaystyle
    \f_i =  \bigvee\limits_{\substack{\vec{s}\in S_1\times\dots\times S_n\\f_i(\vec{s}) = b}} \bigwedge\limits_{k \in N} \delta_{\frac{s_{k}}{m}}(v_{k}),
    $

where $\delta_x(v)$ is a characteristic formula of $x$ in~$\mathbfitL_m$.
(Recall from Lemma~\ref{l:Luk-folklore} that there is a characteristic formula for each element of~$\mathbfitL_m$.)
\item $c_i(s_i) = \frac{s_i}{m}$ for each $s_i\in S_i$.
\item $g(x) = (b-a)x + a$.
\end{itemize}
\end{proposition}

\begin{proof}
The proof is analogous to that of Proposition~\ref{prop:ab_i}.
Besides small structural differences---the current game is only weakly expressible and  basic---we only alter (actually simplify) the key observation for any strategy profile $\vec{s}$:
$$
e\Bigl(\bigwedge\limits_{k \in N} \delta_{\frac{s_k}{m}}(v_{k})\Bigr) =
\begin{cases}
 1 & \text{if } \tuple{e(v_1),\dots, e(v_n)} = \bigl\langle\frac{s_1}{m},\dots, \frac{s_n}{m}\bigr\rangle = \vec{c}(\vec{s}); \\
 0 & \text{otherwise.}
\end{cases}
$$
Thus we still have $e(\f_i(\vec{v})) =1 = \f^{\mathbfitL_m}_i(\vec{c}(\vec{s}))$ iff  $f_i(\vec{s}) = b$.
\end{proof}

Of course, one can combine the previous two approaches and work with any number of truth values between $2$ and $\max_{i\in N}|S_i|$ at the price of using more variables, and we can render it in different logics. Let us now formalize these ideas in the next proposition (note that Propositions~\ref{prop:ab_i} and~\ref{prop:ab_ii} are its corollaries).

\begin{proposition}  \label{prop:ab_iii}
Let $\Gcal =\tuple{N, \{S_{i}\mid i\in N\},\{f_{i}\mid i\in N\}}$ be a finite
strategic game such that the union of ranges of all $f_i$ is $\{a,b\}$,
where $a < b$. Let $m\geq 1$ and $n_i$ be the least integer such that $|S_i| \leq {(m+1)^{n_i}}$ and let us, for a given $s_i\in S_i$ denote  the $(m+1)$-ary representation of $s_i$ by $\tuple{s_i^1, \dots, s_i^{n_i}}\in  \{0,\dots,m\}^{n_i}$.

Moreover let $\alg{A}$ be a standard algebra with distinct elements $x_{0}, \dots, x_m\in A$ such that for each $i\leq m$ there is a characteristic formula $\delta_i(v)$ for $x_i$ in~\alg A.

Then $\Gcal$ is affinely represented by a weakly expressible logical $\alg{A}$-game
$\hat{\Gcal} = \langle N,V,\allowbreak\{V_{i}\mid{i\in N}\},\allowbreak\{S'_{i}\mid{i\in N}\},\allowbreak\{\f_{i}\mid{i\in N}\}\rangle$
via $g$ and $\vec c=\tuple{c_i}_{i\in N}$, where for each $i\in N$:
\begin{itemize}
\item $V_i = \{\tuple{v_i^1, \dots, v_i^{n_i}}\}$ (thus $V=\bigcup_{i\in N}V_i$).
\item $S'_i = \bigl\{\bigl\langle x_{s_i^1}, \dots, x_{s_i^{n_i}}\bigr\rangle \mid s_i\in S_i\bigr\}$.
\item $\displaystyle
    \f_i =  \bigvee\limits_{\substack{\vec{s}\in S_1\times\dots\times S_n\\f_i(\vec{s}) = b}} \bigwedge\limits_{k \in N}\bigwedge\limits_{j \leq n_k} \delta_{s^j_{k}}(v^j_{k}).
    $
\item $c_i(s_i) = \bigl\langle x_{s_i^1}, \dots, x_{s_i^{n_i}}\bigr\rangle$ for each $s_i\in S_i$.
\item $g(x) = (b-a)x + a$.
\end{itemize}
Furthermore, the representing game is basic iff $m \geq  \max_{i\in N}|S_i| -1$.
\end{proposition}

\begin{proof}
A straightforward combination of the proofs of Propositions~\ref{prop:ab_i} and~\ref{prop:ab_ii}.
\end{proof}

Now we leave games with binary payoffs and deal with the strategic games with any finite number
$r$ of possible payoff values. To achieve a more digestible presentation,
 we first formulate a result for a fixed infinitely-valued logic and basic games;
 only then a general variant is presented.
On the other hand, the less general version provides an \emph{affine} representation,
which cannot be guaranteed in the general case.
For simplicity we use the logic given by the algebra $[0,1]^\tri_{\Q\mathrm{G}}$
(see Example~\ref{ex:standard_algebras}); indeed in this algebra we have both the corresponding truth constant $\bar{a}$ and a characteristic formula $\delta_a$ for each rational $a$
and thus each logical $[0,1]^\tri_{\Q\mathrm{G}}$-game $\Gcal$ where $\ES{\alg{A}}{\Gcal}\subseteq \mathbb{Q}$ is expressible.
Notice that by Proposition~\ref{p:Subreducts1}, we can as well use the logic of any algebra expanding $[0,1]^\tri_{\Q\mathrm{G}}$, for example $[0,1]^\tri_{\Q\mathrmL}$, $[0,1]^\tri_{\Q\prl}$,
or $[0,1]_\LPih$ (as the connectives of $[0,1]_{\mathrm{G}}$ are definable in all these algebras).

\begin{proposition} \label{prop:vi}
Let $\Gcal =\tuple{N, \{S_{i}\mid i\in N\},\{f_{i}\mid i\in N\}}$ be a finite strategic game
such that the union of the ranges of all $f_i$ is a set of rational numbers $\{o_1,\ldots,o_r\}$, where
$o_1  < o_2 < \dots <o_r$, and let $m = \max_{i\in N}|S_i| -1$.

Then $\Gcal$ is affinely represented by a \emph{basic} expressible logical\/ $[0,1]^\tri_{\Q\mathrm{G}}$-game
$\hat{\Gcal} = \langle N,V,\allowbreak\{V_{i}\mid{i\in N}\},\allowbreak\{S'_{i}\mid{i\in N}\},\allowbreak\{\f_{i}\mid{i\in N}\}\rangle$
via $g$ and $\vec c=\tuple{c_i}_{i\in N}$,
where for each $i\in N$:
\begin{itemize}
\item $V_i = \{v_i\}$ (thus $V=\{v_1,\dots,v_n\}$).
\item $S'_i = \{\frac{s_i}{m}\mid s_i\in S_i\}$.
\item $\displaystyle
    \f_i =  \bigvee\limits_{\vec{s}\in S_1\times\dots\times S_n} \Bigl( \overline{g^{-1}(f_i(\vec{s}))} \wedge \bigwedge\limits_{k \leq n} \delta_{\frac{s_k}{m}}(v_k)\Bigr),$

    where $\delta_{a}(v)$ is a characteristic formula of $a$ in $[0,1]^\tri_{\Q\mathrm{G}}$.

\item $c_i(s_i) =  \frac{s_i}{m}$ for each $s_i\in S_i$.
\item $g(x) = (o_r-o_1)x + o_1$.
\end{itemize}
\end{proposition}

\begin{proof}
First observe that $\hat{\Gcal}$ is indeed a logical $[0,1]^\tri_{\Q\mathrm{G}}$-game: notice that
$g^{-1}(f_i(\vec{s}))$ is a rational number, as $f_i(\vec{s})$ is rational and $g\colon[0,1]\to\Rc$ is an affine function with rational coefficients; thus there is a corresponding truth constant and $\f_i$ is indeed an $\lang{L}_{[0,1]^\tri_{\Q\mathrm{G}}}$-formula.
Moreover, $\ES{[0,1]^\tri_{\Q\mathrm{G}}}{\hat{\Gcal}} \subseteq [0,1]\cap\Q$ and therefore $\hat{\Gcal}$ is expressible (recall that all rational truth constants are present in the language $\lang{L}_{[0,1]^\tri_{\Q\mathrm{G}}}$). Clearly,
each $c_i$ is a bijection and $g$ is an affine strictly increasing function with rational coefficients. It remains to check that for each strategy profile~$\vec{s}$,
$$
f_i(\vec{s}) = g\bigl(\phi^{[0,1]^\tri_{\Q\mathrm{G}}}_i(\vec c(\vec s))\bigr).
$$
Observe that for any strategy profile $\vec{s}$ we have:
$$
e\Bigl(\overline{g^{-1}(f_i(\vec{s}))} \wedge\bigwedge\limits_{k \leq n} \delta_{\frac{s_k}{m}}(v_k)\Bigr) =
\begin{cases}
 g^{-1}(f_i(\vec{s})) & \text{if } \tuple{e(v_1),\dots, e(v_n)} = \tuple{\frac{s_1}{m},\dots, \frac{s_n}{m}}= \vec{c}(\vec{s}); \\
 0 & \text{otherwise.}
\end{cases}
$$
Therefore we have $\f^{[0,1]^\tri_{\Q\mathrm{G}}}_i(\vec{c}(\vec{s})) = g^{-1}(f_i(\vec{s}))$, as required.
\end{proof}

Furthermore, finite strategic games with rational payoffs can also be affinely represented by logical $\alg A$-games for sufficiently expressive \emph{finite} algebras~\alg A. For simplicity, the following proposition is formulated for finite standard G-algebras with truth constants and $\tri$ (see Example~\ref{ex:standard_algebras}); i.e., $\alg A=\alg G_m^{c\,\tri}$ for sufficiently large $m$. However, Proposition~\ref{p:Subreducts1} again ensures that all expansions of $\alg G_m^{c\,\tri}$ can be used as well (notice that in particular, $\mathbfitL_m^c$ falls within this class, since all connectives of $\alg G_m^{c\,\tri}$ are definable in~$\mathbfitL_m^c$).

\begin{proposition} \label{prop:vi-Gmc}
Let $\Gcal =\tuple{N, \{S_{i}\mid i\in N\},\{f_{i}\mid i\in N\}}$ be a finite strategic game
such that the union of the ranges of all $f_i$ is a set of rational numbers
$\bigl\{\frac{p_1}q,\ldots,\frac{p_r}q\bigr\}$, where $q, p_1, \dots p_r$ are integers and
$p_1  < p_2 < \dots <p_r$. Let $m$ be a natural number such that $m \ge \max\{p_r-p_1,|S_1|,\dots,|S_n|\}-1$.

Then $\Gcal$ is affinely represented by a basic expressible logical\/ $\alg G_m^{c\,\tri}$-game
$\hat{\Gcal} = \langle N,V,\allowbreak\{V_{i}\mid{i\in N}\},\allowbreak\{S'_{i}\mid{i\in N}\},\allowbreak\{\f_{i}\mid{i\in N}\}\rangle$
via $g$ and $\vec c=\tuple{c_i}_{i\in N}$,
where for each $i\in N$:
\begin{itemize}
\item $V_i$, $S'_i$, $\f_i$, and $c_i$ are defined as in Proposition~\ref{prop:vi},
    using $\delta_{a}(v)=\tri((v\rightarrow\bar a)\wedge(\bar a\rightarrow v))$.
\item $g(x) = (mx+p_1)/q $. 
\end{itemize}
\end{proposition}

The proof of Proposition~\ref{prop:vi-Gmc} is essentially the same as the proof of Proposition~\ref{prop:vi},
therefore we omit it.
As shown by the next Proposition~\ref{prop:vi-Lm}, the presence of truth constants can be avoided in standard MV-chains $\mathbfitL_m$ of suitable lengths, at the price of having a slightly larger algebra and only \emph{weak} expressibility of the representing game. The formalizability of the payoff function in $\mathbfitL_m$ is based on the following lemma:

\begin{lemma}\label{l:Lm-zeta}
Let $m$ be a prime number, $a,b\in\mathbfitL_m$ and $a\notin\{0,1\}$.
Then there is an $\lang L_{\mathbfitL_m}$-formula $\zeta_{m,a,b}(v)$
such that $\zeta_{m,a,b}^{\mathbfitL_m}(a)=b$.
\end{lemma}

\begin{proof}
Let $a=\frac pm$ and $b=\frac qm$.
Since $m$ is prime, by Lemma~\ref{l:Luk-folklore}(\ref{l:Luk-folklore-prime}) there is a formula $\xi_{p,m}(v)$ such that $\xi_{p,m}^{[0,1]_{\mathrmL}}\bigl(\frac pm\bigr)=\frac1m$.
Since $\mathbfitL_m$ is a subalgebra of $[0,1]_{\mathrmL}$, we obtain $\xi_{p,m}^{\mathbfitL_m}\bigl(\frac pm\bigr)=\frac1m$ as well;
thus it is sufficient to take $\bigoplus_{i=1}^q\xi_{p,m}$ for $\zeta_{m,a,b}$.
\end{proof}

\begin{proposition} \label{prop:vi-Lm}
Let $\Gcal =\tuple{N, \{S_{i}\mid i\in N\},\{f_{i}\mid i\in N\}}$ be a finite strategic game
such that the union of the ranges of all $f_i$ is a set of rational numbers
$\bigl\{\frac{p_1}q,\ldots,\frac{p_r}q\bigr\}$, where $q, p_1, \dots p_r$ are integers and
$p_1  < p_2 < \dots <p_r$. Let $m$ be a prime number such that
$m \ge \max\bigl(p_r-p_1,|S_1|+1,\dots,|S_n|+1\bigr)$.

Then $\Gcal$ is affinely represented by a basic weakly expressible logical\/ $\mathbfitL_m$-game
$\hat{\Gcal} = \langle N,V,\allowbreak\{V_{i}\mid{i\in N}\},\allowbreak\{S'_{i}\mid{i\in N}\},\allowbreak\{\f_{i}\mid{i\in N}\}\rangle$
via $g$ and $\vec c=\tuple{c_i}_{i\in N}$,
where for each $i\in N$:
\begin{itemize}
\item $V_i = \{v_i\}$ (thus $V=\{v_1,\dots,v_n\}$).
\item $S'_i = \{\frac{s_i+1}{m}\mid s_i\in S_i\}$.
\item $\displaystyle
    \f_i =  \bigvee\limits_{\vec{s}\in S_1\times\dots\times S_n} \Bigl(
    \zeta_{m,\frac{s_i+1}m,g^{-1}(f_i(\vec{s}))}(v_i)
    \wedge \bigwedge\limits_{k \leq n} \delta_{\frac{s_k+1}{m}}(v_k)\Bigr),$

    where $\zeta_{m,a,b}(v)$ is the formula from Lemma~\ref{l:Lm-zeta}
    and $\delta_{a}(v)$ is a characteristic formula of $a$ in~$\mathbfitL_m$.

\item $c_i(s_i) =  \frac{s_i+1}{m}$ for each $s_i\in S_i$.
\item $g(x) = (mx+p_1)/q $. 
\end{itemize}
\end{proposition}

The proof of Proposition~\ref{prop:vi-Lm} is analogous to that of Proposition~\ref{prop:vi}
(just observe that by Lemma~\ref{l:Lm-zeta},
$\zeta_{m,\frac{s_i+1}m,g^{-1}(f_i(\vec{s}))}^{\mathbfitL_m}(v_i)=g^{-1}(f_i(\vec{s}))$).

\begin{example}\label{ex:adapt-finite}
By Propositions~\ref{prop:vi}--\ref{prop:vi-Lm},
the strategic game $\mathcal{NT}$ of Example~\ref{ex:adapt} can be affinely represented not only as the logical $\mathbfitL_4^c$-game $\mathcal{NT}_{\!\mathbfitL_4^c}$ of Examples~\ref{ex:adapt-QL} and~\ref{ex:adapt-representation},
but also, e.g., as a logical $[0,1]^\tri_{\Q\mathrm{G}}$-game %$\mathcal{NT}^\tri_{\Q\mathrm{G}}$
(by Proposition~\ref{prop:vi}),
a logical $\alg G_4^{c\,\tri}$-game 
(by Proposition~\ref{prop:vi-Gmc}),
or a logical $\mathbfitL_7$-game 
(by Proposition~\ref{prop:vi-Lm}).
Proposition~\ref{p:Subreducts1} and the variability of $m$ in Propositions~\ref{prop:vi-Gmc} and~\ref{prop:vi-Lm} admit further algebras for logical representation of $\mathcal{NT}$, including, e.g., $[0,1]_{\Q\mathrmL}$, $[0,1]_\LPih$, $\mathbfitL_5^c$, $\mathbfitL_6^c$, $\mathbfitL_{11}$, $\mathbfitL_{13}$, etc.
In all these cases, the representation is affine and the representing games are basic and expressible
(or weakly expressible in the case of $\mathbfitL_m$-games).

Notice that while the payoff formulas produced by Propositions~\ref{prop:vi}--\ref{prop:vi-Lm} are rather large, a much more compact logical representation of $\mathcal{NT}$ exists in algebras that contain $\mathbfitL_4^c$ as a subreduct: see the formulas $\f_i$ in Example~\ref{ex:adapt-QL}.
Notice also that despite the representability as an $\mathbfitL_4^c$- or $\mathbfitL_7$-game,
$\mathcal{NT}$ cannot be represented as a finite {\L}ukasiewicz game of Marchioni and Wooldridge, since the sets of strategies and payoff values have different cardinalities.
\end{example}

As  already mentioned above, the following general version comes at the price of
loosing the affinity of representations.

\begin{proposition} \label{prop:vii}
Let $\Gcal =\tuple{N, \{S_{i}\mid i\in N\},\{f_{i}\mid i\in N\}}$ be a finite game such that the union of ranges of all $f_i$ is $\{o_1,\ldots,o_r\}$, where $o_1  < o_2 < \dots <o_r$, and let $m = \max_{i\in N} |S_i| -1 $. Let $\alg{A}$ be an arbitrary standard algebra with distinct elements $a_0, \dots, a_{m-1}$ and distinct elements $b_1< \dots < b_r$ (the two sets can overlap, though) and such that there are characteristic formulas $\delta_i$ in \alg A for each $a_i$, $i < m$, and (definable) truth constants $\bar b_i$ for each $b_i$, $i\leq r$.

Then $\Gcal$ is represented by a \emph{basic} weakly expressible $\alg{A}$-logical game
$\hat{\Gcal} = \langle N,V,\allowbreak\{V_{i}\mid i\in N\},\allowbreak\{S'_{i}\mid i\in N\},\allowbreak\{\f_{i}\mid i\in N\}\rangle$
via $g$ and $\vec c=\tuple{c_i}_{i\in N}$,
where for each $i\in N$:
\begin{itemize}
\item $V_i = \{v_i\}$  (thus $V=\{v_1,\dots,v_n\}$).
\item $S'_i = \{a_{s_i}\mid s_i\in S_i\}$.
\item $\displaystyle
    \f_i =  \bigvee\limits_{\vec{s}\in S_1\times\dots\times S_n} \Bigl( \overline{g^{-1}(f_i(\vec{s}))} \wedge \bigwedge\limits_{k \leq n} \delta_{s_k}(v_k)\Bigr)
    $.
\item $c_i(s_i) =  a_{s_i}$  for each $s_i\in S_i$.
\item $g\colon[0,1]\to\Rc$ is a strictly increasing function with $g(b_j) = o_j$ for every $j\leq r$.
\end{itemize}
\end{proposition}

\begin{proof}
Observe that $\hat{\Gcal}$ is a logical $\alg{A}$-game; moreover, $\ES{\alg{A}}{\hat{\Gcal}} =  \{a_1, \dots, a_m\}$  and thus $\hat{\Gcal}$ is weakly expressible (by the theorem's assumptions). Clearly each $c_i$ is a bijection and thus it remains to check that for each strategy profile $\vec{s}$ we have
$$
f_i(\vec{s}) = g(\phi_i(\vec c(\vec s))).
$$
As before, it suffices to observe that for any strategy profile $\vec{s}$ we have
$$
e\Bigl(\overline{g^{-1}(f_i(\vec{s}))} \wedge \bigwedge\limits_{k \leq n} \delta_{s_k}(v_k)\Bigr) =
\begin{cases}
 b_{f_i(\vec{s})} & \text{if } \tuple{e(v_1),\dots, e(v_n)} = \tuple{a_{s_1},\dots, a_{s_n}} = \vec{c}(\vec{s}); \\
 0 & \text{otherwise.}
\end{cases}
$$
We obtain $\f_i^{\alg A}(\vec{c}(\vec{s})) = b_{f_i(\vec{s})}$ and so $g(\f^\alg{A}_i(\vec{c}(\vec{s}))) = g(b_{f_i(\vec{s})}) = f_i(\vec{s})$.
\end{proof}

Just like for the combination of Proposition~\ref{prop:ab_i} and~\ref{prop:ab_ii} into
Proposition~\ref{prop:ab_iii}, we could also put together Proposition~\ref{prop:vi} and~\ref{prop:vii},
but refrain from doing so here.

\section{Expressing equilibria of logical games}\label{sec:ExpressingNE}

In this section we show how pure and mixed Nash equilibria of logical games can be expressed by propositional formulas, under particular
conditions.

Recall from Lemma~\ref{t:Representation} that whenever a logical game represents a strategic game, it has the same pure equilibria (modulo the representation); and if the game is finite and the representation is affine, then even mixed equilibria are preserved by the representation. Consequently the formulas derived below characterize equilibria not only in logical games themselves, but also in the strategic games they may represent.

Throughout this section we use the notation $\Gcal = \langle N,V,\allowbreak\{V_{i}\mid{i\in N}\},\allowbreak\{S_{i}\mid{i\in N}\},\allowbreak\{\f_{i}\mid{i\in N}\}\rangle$ for any finite logical $\alg{A}$-game, where  $\alg{A}$ is a standard algebra. Furthermore, let $S = S_1\times\dots\times S_n$.

\subsection{Pure Nash equilibria} \label{sec:pureNE}

A crucial  observation is the fact that in (weakly) expressible games
each player's choice of a strategy can be  encoded by $\lang{L}_\alg{A}$-formulas.
In this section we show how this fact can be employed to express by an $\lang{L}_\alg{A}$-formula
that a certain strategy profile is a pure Nash equilibrium. Consequently, as we will show, it can
also be expressed that such an equilibrium exists. For simplicity we start with expressible games and deal with the more complicated case of weakly expressible ones later. Recall that in these games we have a truth constant $\bar a$ for each $a \in \ES{\alg{A}}{\Gcal}$.

Let us consider auxiliary variables  $\vec{w} = \tuple{w^1, w^2,\dots, w^{\max_{i \in N} |V_i|}}$ different from those in $V$ and define formulas $\gamma_{i}(\vec{v},\vec{w})$, for each $i\in N$, and $\gamma_{\Gcal}(\vec{v})$ as follows:
\begin{align*}
    \gamma_{i} &= \f_{i}(\vec{v}_{1},\dotsc,\vec{v}_{i-1},w^1, \dots, w^{|V_i|}, \vec v_{i+1},\dotsc,\vec{v}_{n}) \rightarrow \f_{i}(\vec{v}_{1},\dotsc,\vec{v}_{n})
\\ 
    \gamma_{\Gcal} &=  \bigwedge_{i\in N}\bigwedge_{\vec{s}_i\in S_i} \gamma_{i}(\vec{v}_{1},\dotsc,\vec{v}_{i-1},\bar s_i^1, \dots, \bar s_i^{|V_i|},\vec v_{i+1},\dotsc,\vec{v}_{n}).
\end{align*}

\begin{lemma}\label{l:CharNashPE}
Let $\Gcal$ be an expressible finite $\alg{A}$-game. Then the following are equivalent for each strategy profile $\vec{s}^*$:
\begin{enumerate}
\item  $\vec{s}^*$ is a pure Nash equilibrium of $\Gcal$.
\item  $\vec{s}^*$ satisfies $\gamma_{\Gcal}(\vec{v})$.
\end{enumerate}
\end{lemma}

\begin{proof}
The statement is a straightforward consequence of the definition of a pure Nash equilibrium and
the properties of the connectives $\wedge$ and $\rightarrow$ ensured by Definition~\ref{def:A} for all standard algebras of truth degrees.
\end{proof}

\begin{theorem}\label{t:ExistenceNashPE}
Let $\Gcal$ be an expressible finite logical $\alg{A}$-game.
Then the following are equivalent:
\begin{enumerate}
\item $\Gcal$ allows for a pure Nash equilibrium.
\item The following formula is satisfiable:
\begin{equation}\label{eq:pNE}
 \Bigl(\bigvee_{\vec{s}\in  S}\, \bigwedge_{i\in N}\bigwedge_{j\leq |V_i|}\bigl(\x_{{s}^j_{i}}(v^j_{i})\bigr)\Bigr) \wedge \gamma_\Gcal(\vec{v}).
\end{equation}
\end{enumerate}
Moreover, if the game $\Gcal$ is full, then \eqref{eq:pNE} can be replaced just by~$\gamma_\Gcal(\vec{v})$.
\end{theorem}

\begin{proof}
The statement follows from Lemma \ref{l:CharNashPE}. It suffices to observe that an evaluation $e$ satisfies the left conjunct
in~\eqref{eq:pNE}  
if and only if $\vec{e}(\vec{v})$ 
is a strategy profile.
\end{proof}

Recall that all Boolean games are expressible. Thus Lemma \ref{l:CharNashPE} and Theorem~\ref{t:ExistenceNashPE} apply to Boolean games~\cite{Harrenstein-HMW:BooleanGames}
in particular.
Likewise, finite {\L}ukasiewicz games are covered. More precisely, Theorem~2
of~\cite{Marchioni-Wooldridge:LukasiewiczGames} amounts to a variant of a particular subcase of Theorem~\ref{t:ExistenceNashPE},
where the underlying algebra is $\mathbfitL_n^c$, only \emph{full} games are considered, and a  somewhat more complex variant of our $\gamma_\Gcal(\vec{v})$ is used. Full $\mathbfitL_n$-games are treated only indirectly in~\cite{Marchioni-Wooldridge:LukasiewiczGames,Marchioni-Wooldridge:LukasiewiczGamesTOCL},
by showing that they are (in our terminology) weakly expressible. This case is covered by Lemma~\ref{l:weakly_NE} and Theorem~\ref{th:weakly_exists_NE}, below.

%FFV
%PC: check the paragraphs above once more
%CF: Allow me to make a remark that is not relevant here: I wonder whether the REALLY interesting parts of [8] -- a calculus for winning strategies, relativized validity and satisfiability, operations on games -- have ever been generalized to many-valued payoffs. In my mind these were at least as interesting topics for future work than what we wrote, so far, in the conclusion.

If the game is just weakly expressible, we have to use a more complex formula instead of $\gamma_{\Gcal}$ to formulate and prove analogues of Lemma \ref{l:CharNashPE} and Theorem \ref{t:ExistenceNashPE}.
In fact, we can keep the definition of the formulas $\gamma_i$, but we have to modify
 $\gamma_{\Gcal}$ to include additional auxiliary variables that correspond to all the elementary strategies
and will play the role of truth constants.

Formally, we introduce new variables $\{q_a \mid a \in \ES{\alg{A}}{\Gcal}\}$, different from those in $V$ and~$\vec{w}$. Note that the set $\ES{\alg{A}}{\Gcal}$ is a subset of $A$ and thus is naturally ordered. Therefore
we can use $\vec{q}$ to unambiguously denote the sequence of those variables.
Now we can define $\gamma'_{\Gcal}$ as a formula over the variables $\vec{v}$ and $\vec{q}$:
\[
  \gamma'_{\Gcal} = \Bigl(\bigwedge_{a\in \ES{\alg{A}}{\Gcal}} \x_a(q_a)\Bigr) \wedge \Bigl(\bigwedge_{i\in N} \bigwedge_{\vec{s}_i\in S_i} \gamma_{i}(\vec{v}_{1},\dotsc,\vec v_{i-1},q_{s_i^1}, \dots q_{s_i^{|V_i|}},\vec v_{i+1},\dotsc,\vec{v}_{k}) \Bigr).
\]
Note that the formula $\bigwedge_{a\in \ES{\alg{A}}{\Gcal}} \x_a(q_a)$ is satisfiable only by those evaluations that map each variable $q_a$ to $a$. In this manner we
obtain the promised generalizations of Lemma~\ref{l:CharNashPE} and Theorem~\ref{t:ExistenceNashPE},
which are then applicable, e.g., to all  $\mathbfitL_m$-logical games and all
$[0,1]_\mathrmL$-logical games with finite sets of (assignment of) rationals as strategies.

\begin{lemma} \label{l:weakly_NE}
Let $\Gcal$ be a weakly expressible finite $\alg{A}$-game. Then the following are equivalent for each strategy profile $\vec{s}^*$:
\begin{enumerate}
\item  $\vec{s}^*$ is a pure Nash equilibrium of $\Gcal$.
\item  The \alg A-evaluation $\vec{s}^\ast, a_1, \dots, a_{|\ES{\alg{A}}{\Gcal}|}$ satisfies $\gamma'_{\Gcal}(\vec{v},\vec{q})$.
\end{enumerate}
\end{lemma}

\begin{theorem} \label{th:weakly_exists_NE}
Let $\Gcal$ be  a weakly expressible finite $\alg{A}$-game. Then the following are equivalent:
\begin{enumerate}
\item $\Gcal$ allows for a pure Nash equilibrium.
\item The following formula is satisfiable:
\begin{equation}\label{eq:mNE}
 \Bigl(\bigvee_{{\vec{s}\in S}}\, \bigwedge_{i\in N}\bigwedge_{j \leq |V_i|}\bigl(\x_{{s}_{i}^j}(v^j_i)\bigr)\Bigr) \wedge \gamma'_\Gcal(\vec{v},\vec{q}).
\end{equation}
\end{enumerate}
Moreover, if the game $\Gcal$ is full, then \eqref{eq:mNE} can be replaced just by the formula~$\gamma'_\Gcal(\vec{v})$.
\end{theorem}

\subsection{Mixed Nash equilibria} \label{sec:mixedNE}

In order to characterize mixed strategy  profiles  we have to express probability distributions and corresponding expected payoffs in a propositional language. On the algebraic side this means  that the additive as well as  multiplicative structure of the real unit interval $[0,1]$ must be employed. Interestingly enough, there are numerous natural examples of many-valued logics that provide such a rich semantics over $[0,1]$:
in particular, the logics of algebras expanding the standard \prl-algebra $[0,1]_{\prl}$ (including $[0,1]^\tri_{\prl}$, $[0,1]_{\Q\prl}$, $[0,1]^\tri_{\Q\prl}$ $[0,1]_\LPi$, or $[0,1]_\LPih$, see Example~\ref{ex:standard_algebras}) fall within this class.

We proceed with a simple lemma, crucial for expressing probability distributions. The underlying idea is based on \MV-algebraic partitions of unity; see, e.g., \cite{Mundici:NonbooleanPartitions}.
We present its proof for the readers convenience.

\begin{lemma}\label{l:Probability-gen}
Let \alg A expand the standard $\MV$-algebra $[0,1]_\mathrmL$.
For every $n\geq 2$ there is an $\lang L_{\alg A}$-formula $\delta(\vectn{p})$ such that an \alg A-evaluation $\vec{a} \in [0,1]^n$ satisfies $\delta$ iff\/ $\sum_{i\leq n} a_{i} = 1$.
\end{lemma}

\begin{proof}
We define
$$
\delta = \Bigl(\bigoplus\limits_{i\leq n} p_{i}\Bigr) \wedge \bigwedge\limits_{i\leq n} \biggr( \Bigl(\bigoplus\limits_{\substack{j\leq n\\j \neq i}} p_{j}\Bigr) \to \neg p_{i}\biggr).
$$
Clearly the satisfiability of the first conjunct implies $\sum_{i \leq n} a_{i} \geq 1$.
Therefore  $a_{i} >0$ for at least one $i\leq n$. Moreover the satisfiability of the second conjunct implies
$$
\sideset{}{^{\alg{\boldsymbol{A}}}}\bigoplus\limits_{\substack{j\leq n\\j \neq i}} a_{j} \leq 1 - a_{i} < 1.
$$
This yields the inequality
$$
\sum\limits_{\substack{j\leq n\\j \neq i}} a_{j} =
\sideset{}{^{\alg{\boldsymbol{A}}}}\bigoplus\limits_{\substack{j\leq n\\j \neq i}} a_{j} \leq 1 - a_{i},
$$
which entails $\sum_{j\leq n} a_{j} \leq 1$.
The converse direction is trivial.
\end{proof}

For each $i\in N$ and each strategy $\vec{s}_i\in S_i$ let us introduce a variable $p_i^{\vec{s}_i}$.
Moreover, let $\vec{p}_i$ denote the tuple $\tuple{p_i^{\vec{s}_i} \mid \vec{s}_i\in S_i}$.
(The tuple is unique with respect to the lexicographic order on $A^{|V_i|}$.)
For every $i \in N$, any $\alg{A}$-evaluation of the variables $\vec{p}_i$
can be thought of as a mapping $S_i\to[0,1]$. This enables us to formulate a particular instance of Lemma~\ref{l:Probability-gen} and thus to obtain:

\begin{lemma}\label{l:Probability}
Let \alg A expand the standard \PL-algebra $[0,1]_\PL$ and $\Gcal$ be a finite expressible logical $\alg{A}$-game. Then for each $i\leq n$ there is a formula $\pd_i(\vec{p}_i)$ such that an evaluation $\pr_i \in [0,1]^{S_i}$ satisfies $\pd_i$ iff\/ $\pr_i$ is a probability distribution over~$S_i$.
\end{lemma}

As a consequence, an element $\pr=\tuple{\pr_1, \dots, \pr_n} \in [0,1]^S$, where each $\pr_i$ satisfies the formula $\pd_i$, can be seen as a mixed strategy profile in the game $\Gcal$.  This enables us to define the expected payoff for a player~$i$ in a finite expressible logical $\alg{A}$-game $\Gcal$ as follows:
$$
\Ex_i(\vec{p}) = \Ex_i(\vec{p}_1,\dots,\vec{p}_n) = \bigoplus\limits_{\mathbf{s}\in S} \Bigl( \f_i(\bar s_1^1, \dots, \bar s_n^{|V_n|}) \odot \bigodot\limits_{j\leq n}p_j^{\vec{s}_j}\Bigr).
$$
Recall that by $\vec{p}_{-i}$ we denote the sequence of variables $\vec{p}$ where the subsequence $\vec{p}_i$ removed; for every pure strategy $\vec{a}_i\in S_i$, by $(\vec{a}_i,\vec{p}_{-i})$ we denote the mixed strategy profile in which the mixed strategy of player $i$ is the Dirac distribution $\delta_{\vec{a}_i}$ concentrated at $\vec{a}_i$. Moreover we define
the expected payoff for a player $i$ in a  mixed strategy profile $(\vec{a}_i,\vec{p}_{-i})$ as follows:
$$
\Ex_i(\vec{a},\vec{p}_{-i}) = \bigoplus\limits_{\substack{\vec{s}\in S\\\vec{s}_i = \vec{a}}}
    \Bigl( \f_i(\bar s_1^1, \dots, \bar s_n^{|V_n|}) \odot \bigodot\limits_{\substack{j\leq n\\j\neq i}}p_j^{\vec{s}_j}\Bigr).
$$

The above definitions and conventions allow us to formulate the following theorem,
which provides the announced  logical characterization of mixed Nash equilibria; its proof is a straightforward consequence of Proposition~\ref{p:Games-MixedNE}.

\begin{theorem}\label{thm:mixedNE}
Let $\alg{A}$ be an algebra expanding $[0,1]_\prl$. Let $\Gcal$ be a finite expressible logical $\alg{A}$-game and let\/ $\pr^\ast\in[0,1]^S$ be a mixed strategy profile in~$\Gcal$. Then the following are equivalent:
\begin{enumerate}
\item $\pr^\ast$ is a mixed Nash equilibrium.
\item $\pr^\ast$ satisfies the following formula:
\begin{equation}\label{eq:mNE-formula}
    \bigwedge_{i\leq n} \Bigl( \pd_i(\vec{p}_i) \wedge
    \bigwedge_{\vec{a}_i\in S_i} \bigl( \Ex_i(\vec{a}_i,\vec{p}_{-i} % \vec{p}_{-i}^\ast
    ) \rightarrow \Ex_i(\vec{p})\bigr)\Bigr).
\end{equation}
\end{enumerate}
\end{theorem}

\begin{example}\label{ex:love-hate-mNE}
Recall the finite variant $\mathcal{LH}$ of \emph{Love and Hate} from Example~\ref{ex:love-hate}. In Example~\ref{ex:love-hate-representation} it was represented as a logical $\mathbfitL_m$-game $\mathcal{LH}_{\mathbfitL_m}$ with the same pure and mixed equilibria (since the transformation was affine via the identity functions $g$ and~$\vec c$). In virtue of Proposition~\ref{p:Subreducts1}, $\mathcal{LH}_{\mathbfitL_m}$ can as well be regarded as a $[0,1]_\prl$-game, which again has the same pure and mixed equilibria.
A routine calculation shows that the $n$-tuple $\pr^\ast$ of the mixed strategies $\tuple{p_1,\dotsc,p_n}$ defined in Example~\ref{ex:love-hate} satisfies the formula~\eqref{eq:mNE-formula} in $[0,1]_{\prl}$,
and thus $\pr^\ast$ is indeed a mixed Nash equilibrium in $\mathcal{LH}$ and $\mathcal{LH}_{\mathbfitL_m}$.
\end{example}

\section{Conclusion}\label{sec:conclusion}

We have taken up a line of research initiated by the introduction of Boolean games in \cite{Harrenstein-HMW:BooleanGames}
and generalized in \cite{Marchioni-Wooldridge:LukasiewiczGames,Marchioni-Wooldridge:LukasiewiczGamesTOCL} to {\L}ukasiewicz games.
These are particular types of strategic games, where the payoff function for each player is specified
by a propositional formula (of classical logic or some {\L}ukasiewicz logic, respectively)  and where each
strategy assigns truth values (Boolean or many-valued) to the propositional variables under control
of the player in question. The scope of general strategic games that can be directly represented
as Boolean or {\L}ukasiewicz games is limited by the specific type of the formulas that represent
the payoff functions. Motivated by this observation, we present a more general approach referring to a wide class of algebras of truth values in the real unit interval.
For any such algebra~$\alg{A}$ there is a corresponding notion of logical $\alg{A}$-game, where
propositional formulas over the corresponding language $\lang{L}_\alg{A}$ express the
players' payoff functions. Based on a formal definition of the representability of a general
strategic game as a logical $\alg{A}$-game, we have shown for several quite general
classes of finite
strategic games how they can be represented as logical $\alg{A}$-games. Furthermore
we have shown that, for sufficiently expressible algebras $\alg{A}$, the existence
of a pure Nash equilibrium in a logical $\alg{A}$-game can be expressed by an
$\alg{A}$-formula.  This is due to the observation that strategy profiles of logical
games can be identified with evaluations (truth-value assignments) and the fact that
the equilibrium conditions  can be expressed by corresponding propositional formulas.
As is well known, finite strategic games always admit a Nash equilibrium in terms
of mixed strategies. However, it has been left open so far whether  such mixed equilibria
can be characterized by propositional formulas. We have taken up this challenge and
proved that,  for sufficiently rich algebras $\alg{A}$, one can encode
probability distributions over strategies (i.e., evaluations) and
find an $\lang{L}_\alg{A}$-formula that is satisfied by an interpretation if and only if that interpretation
encodes a mixed Nash equilibrium.

Several directions for future research seem natural. For example, we did not consider complexity
issues in this paper. Moreover, from a logical point of view, a particularly interesting question arises for mixed
equilibria in infinite games: The logical machinery developed here is quite obviously insufficiently
expressive to deal with probability distributions over infinite sets of strategies. However, we conjecture
that our results can be generalized also to that case by employing quantified propositional logics
(see, e.g.,~\cite{Baaz-Fermuller-Veith:AnalyticCalculusQuantified}).

Another option would be to deepen the research on the game-theoretic side with the goal of characterizing the special structure of Nash equilibria associated with certain classes of infinite games in which the payoffs are continuous and the strategy space of each player is
isomorphic to $[0,1]$. The motivation comes from known results for separable (polynomial) games \cite[Chapter 11]{Dresher:GamesStrategy}: Every continuous game that falls within this class has at least one mixed Nash equilibrium whose components are probability measures with finite supports. Specifically, let us suppose that each player $i\in N$ has $[0,1]$ as her strategy space. The payoff functions $f_i\colon [0,1]^n\to [0,1]$ are assumed to be polynomials in $n$ variables. Although the set of all mixed strategies in this game is the set of all Borel probability measures over $[0,1]$, which is far beyond the scope of any straightforward logic-based treatment,
it is known that every polynomial game has a Nash equilibrium consisting of finite mixed strategies $p_i$ only. Namely for every $i\in N$ there exist strategies $s_i^1,\dots,s_i^m\in [0,1]$ and coefficients $\alpha_i^1,\dots,\alpha_i^m\geq 0$ such that $\sum_{j=1}^m \alpha_i^j =1$ and
	\begin{equation}\label{eq:fin}
	p_i=\sum_{j=1}^m \alpha_i^j \cdot \delta_{s_i^j},
	\end{equation}
where each $\delta_{s_i^j}$ is the Dirac probability distribution at $s_i^j$. In \cite{Kroupa-Majer:StrategicReasoning} a  sufficient condition for the existence of finite mixed equilibria in constant-sum games given by McNaughton functions has been given. This means that the scope of our logical analysis of mixed equilibria (Section \ref{sec:mixedNE}) can possibly be expanded further, capturing also the existence of equilibria of the type (\ref{eq:fin}) directly in a sufficiently strong propositional language $\lang{L}_\alg{A}$.

Finally let us draw attention to the fact, that when Harrenstein \textit{et al.} introduced
Boolean games in \cite{Harrenstein-HMW:BooleanGames},  they addressed
several topics that go beyond the mere representability of certain
games by logical formulas. In particular they also considered operations on games, a
form of relativized validity and satisfiability motivated by their games,  and a
calculus for deriving winning strategies. It would be certainly interesting to see to what extent these
and related topics can be developed also in the considerably more general many-valued setting
presented here.

\end{document}